\documentclass[12pt]{article}
\usepackage{amsmath,amsfonts,amsthm,amssymb,bbm}
\usepackage{fullpage}
\usepackage[shortlabels]{enumitem}
\usepackage[utf8]{inputenc}
\usepackage{xcolor}
\usepackage{biblatex}
\usepackage[pdfencoding=auto, psdextra]{hyperref}

\newcommand{\jbrac}[1]{\left\langle {#1}\right\rangle}
\newcommand{\norm}[1]{\left\Vert {#1}\right\Vert}
\newcommand{\pare}[1]{\left({#1}\right)}
\newcommand{\sqbrac}[1]{\left[ {#1}\right]}
\newcommand{\brac}[1]{\left\{{#1}\right\}}
\newcommand{\R}{\mathbb{R}}
\newcommand{\D}{\mathbb{D}}
\newcommand{\E}{\mathbb{E}}
\newcommand{\Z}{\mathcal{Z}}
\newcommand{\bQ}{\mathbf{Q}}

\newenvironment{bracket}{\left\{\begin{aligned}}{\end{aligned}\right.}

\newtheorem{theorem}{Theorem}[section]
\newtheorem{proposition}[theorem]{Proposition}
\newtheorem{lemma}[theorem]{Lemma}
\newtheorem{corollary}[theorem]{Corollary}
\theoremstyle{definition}

\theoremstyle{remark}
\newtheorem{remark}[theorem]{Remark}

\title{Optimal mass normalizability for Gibbs measure associated with NLS on the 2D disc}
\author{Tianhao Xian}
\date{}

\begin{document}
	\maketitle
	
	\begin{abstract}
		We prove the normalizability of Gibbs measure associated with radial focusing nonlinear Schr\"{o}dinger equation (NLS) on the 2-dimensional disc $\D$, at critical mass threshold. The result completes the study of optimal mass normalizability on $\D$ by Oh-Sosoe-Tolomeo (2021).
	\end{abstract}
	
	\section{Introduction}
	In this paper, we complete Oh-Sosoe-Tolomeo's study \cite{oh2021optimal} of optimal normalizability threshold for the Gibbs measure for the focusing nonlinear Schr\"{o}dinger equation (NLS), extending their result to the setting of $L^2_{rad}(\D)$, the square integrable radial functions on unit disc.
	
	The focusing nonlinear Schr\"{o}dinger equation
	\begin{equation}
		\begin{bracket}
			i\partial_t u(t,x) + \Delta u(t,x) + |u(t,x)|^{p-2}u(t,x) = 0,\\
			u(0,\cdot) = u_0
		\end{bracket}
	\end{equation}
	is an evolution equation corresponding to the Hamiltonian
	\begin{equation*}
		H(u) = \frac{1}{2}\int_\D |\nabla u|^2 - \frac{1}{p}\int_\D |u|^{p}.
	\end{equation*}
	The Gibbs measure is a probability on  function space formally defined as a weighted Lebesgue measure
	\begin{equation} \label{def: formal gibbs measure}
		d\rho = \Z^{-1}e^{-H(u)}du,
	\end{equation}
	where $\Z$ is a normalization constant known as the partition function. The conservation of the Hamiltonian $H$ (which holds for smooth data) suggests that $\rho$ should be invariant under the flow of NLS. Since Lebowitz-Rose-Speer's construction of the Gibbs measure \cite{lebowitz1988statistical}, there have been numerous studies of the invariance of Gibbs measure with respect to the flow of NLS. (In particular, this yields almost surely global well-posedness of NLS.) McKean \cite{mckean1995statistical} proved invariance of the Gibbs measure for NLS on $\mathbb{T}$. Meanwhile, Bourgain \cite{bourgain1994periodic} proved the same result with a more analytic method, combining the deterministic local well-posedness with the invariant measure for the truncated NLS. Later, he constructed an invariant Gibbs measure for a modified NLS equation on $\mathbb{T}^2$, where the local well-posedness used is probabilistic \cite{bourgain1996invariant}. Bourgain's method was then applied in other settings. Tzvetkov constructed an invariant measure for NLS on the unit disc in $\R^2$, with subcritical nonlinearity \cite{tzvetkov2006invariant}. This result was improved by Bourgain-Bulut to include the critical nonlinearity with small mass initial data \cite{bourgain2014almost1}. The restriction on the mass was imposed due to the lack of an optimal normalization result of the Gibbs measure for focusing NLS, which is the goal of this paper. For the defocusing NLS, the Gibbs measure can be constructed without restriction on the mass. Tzvetkov proved invariance of the defocusing sub-quintic NLS \cite{tzvetkov2008invariant} and Bourgain-Bulut for defocusing cubic NLS on 3$d$ unit ball \cite{bourgain2014almost2}.
	
	Since a translation invariant measure on infinite dimensional space cannot be locally finite, the definition of the Gibbs measure in (\ref{def: formal gibbs measure}) is formal. To rigorously construct the Gibbs measure $\rho$, let $\brac{e_n}$ be the orthonormal basis of $L^2_{rad}(\D)$ consisting of eigenfunctions of $-\Delta$, with corresponding eigenvalues $\lambda_n$. Let $\{g_n\}_{n\geq 1}$ be a sequence of independent standard complex-valued Gaussian\footnote{That is, $\Im g_n$ and $\Re g_n$ are independent with law $\mathcal{N}(0,1/2)$.}. Define the free Gaussian measure $\mu$ on $L^2_{rad}(\D)$ as the law of random variable
	\begin{equation}\label{def: u series}
		u(\omega) = \sum_n \frac{g_n(\omega)}{\sqrt{\lambda_n}}e_n.
	\end{equation}
	In other words, $\mu$ is the Gaussian measure on $L^2_{rad}(\D)$ with mean $0$ and covariance $(-\Delta)^{-1}$.\footnote{$(-\Delta)^{-1}$ is the solution map of the Poisson equation $-\Delta u = f$, $f\in L^2_{rad}(\D)$, with Dirichlet boundary conditions.}
	
	Writing $u = \sum_n u_n e_n$, then
	\begin{equation*}
		d\mu(u) = \prod_n \frac{1}{2\pi\lambda_n}e^{-\frac{1}{2}|u_n|^2} d\Re u_n d\Im u_n =``\Z_0^{-1}e^{-\frac{1}{2}\int_\D |\nabla u|^2} du".
	\end{equation*}
	The Gibbs measure $\rho$ is then defined as a weighted $\mu$-measure with density $e^{\frac{1}{p+1}\int|u|^{p+1}}$. Unfortunately, this density is not $\mu$-integrable \cite{lebowitz1988statistical}. A remedy is to restrict the measure to a smaller set expected to be invariant under the equation. In \cite{oh2010invariant}, Oh-Quastel constructed a Gibbs measure on $\mathbb{T}$, conditioning on fixed mass $\int |u|^2$ and momentum $\int iu\bar{u}_x$. A more commonly used method is $L^2$-truncation: define
	\begin{equation}\label{def: gibbs measure}
		d\rho = d\rho_{K,p} = \Z_{K,p}^{-1} \mathbbm{1}_{\norm{u}_{L^2}\leq K}e^{\frac{1}{p+1}\int|u|^{p+1}}d\mu.
	\end{equation}
	To guarantee $\rho_{K,p}$ is indeed a probability measure, we need to show the density is $\mu$-integrable, or equivalently, the partition function $\mathcal{Z}_{k,p}$ is finite.
	
	The study of the integrability of $\rho_{K,p}$ was initiated by Lebowitz-Rose-Speer \cite{lebowitz1988statistical}, where they considered the Gibbs measure on $L^2(\mathbb{T})$. Previous results in the torus setting \cite{lebowitz1988statistical, bourgain1994periodic, oh2021optimal} are summarized as following
	\begin{theorem}[Focusing Gibbs measure on torus]
		For $p\geq 2$, $K>0$, the partition function of Gibbs measure for focusing NLS,
		\begin{equation*}
			\Z_{K,p} = \E_{\mu}\sqbrac{e^{\frac{1}{p}\int_{\mathbb{T}}|u|^{p} dx},\norm{u}_{L^2(\mathbb{T})}\leq K},
		\end{equation*}
		satisfies
		\begin{enumerate}[(a)]
			\item(Subcritical) For $p<6$, $\Z_{K,p} < \infty$.
			\item(Supercritical) For $p>6$, $\Z_{K,p} = \infty$.
			\item(Critical) For $p=6$, $\Z_{K,p} < \infty$ if and only if $K\leq \norm{Q}_{L^2(\R)}$, where $Q$ is the ground state solution (positive, radial, decreasing to 0 at infinity) of equation
			\begin{equation*}
				-Q'' + Q - Q^5 = 0.
			\end{equation*}
		\end{enumerate}
	\end{theorem}
	In \cite{lebowitz1988statistical}, Lebowitz-Rose-Speer proved (a), (b). For the critical case $p=6$, they showed that $\Z_{K,6} < \infty$ if $K$ is small enough. Later Bourgain \cite{bourgain1994periodic} reproved their results using the series expression (\ref{def: u series}). Part (c) was completed by Oh-Sosoe-Tolomeo \cite{oh2021optimal}.
	
	Similar methods apply to the $L^2_{rad}(\D)$ setting, where the critical nonlinearity is $p=4$. In \cite{tzvetkov2006invariant}, Tzvetkov constructed the measure and proved invariance for the subcritical case $p<4$. Later, Bourgain-Bulut \cite{bourgain2014almost1} obtained a proof  for the critical case with small $L^2$ cutoff $K$. The optimal cutoff for $K$ is $\norm{Q}_{L^2(\R^2)}$, where $Q$ is the ground state solution of the equation
	\begin{equation*}
		-\Delta Q + Q - Q^3 = 0.
	\end{equation*}
	Oh-Sosoe-Tolomeo \cite{oh2021optimal} proved integrability for $K<\norm{Q}_{L^2}$ and non-integrability for $K>\norm{Q}_{L^2}$. The goal of this paper is to complete the remaining case $K = \norm{Q}_{L^2}$.
	\begin{theorem}[Main result]\label{thm: main result}
		The partition function for the Gibbs measure (\ref{def: gibbs measure}) at the critical threshold $p=4, K = \norm{Q}_{L^2(\R^2)}$ is finite. More precisely,
		\begin{equation*}
			\mathcal{Z} := \E_\mu \sqbrac{e^{\frac{1}{p}\int_{\D}|u|^{p} dx},\norm{u}_{L^2(\D)}\leq \norm{Q}_{L^2(\R^2)}}<\infty.
		\end{equation*}
	\end{theorem}
	
	\subsection{Strategy of the proof}
	For now on, we discuss the partition function  at critical nonlinearity and integrability threshold: $p=4, K = \norm{Q}_{L^2(\R^2)}$. Looking at the formal density of $\rho$ in (\ref{def: formal gibbs measure}), one sees that the key part of the proof is to bound the Hamiltonian $H$ from below. The following strategy was introduced by Oh-Sosoe-Tolomeo \cite{oh2021optimal}.
	
	The largest part of the density $e^{-H}$ occurs at minimum points of $H$, if they exist. Although $\mu$ is a measure on $L^2_{rad}(\D)$, by a density argument, we can strict our discussion to a dense subset $H_{rad,0}^1(\D)$, the set of radial $H^1(\D)$ functions vanishing on $\partial \D$. It turns out that, conditioning on $\brac{\norm{u}_{L^2(\D)}\leq \norm{Q}_{L^2(\R^2)}}$, the Hamiltonian $H$ has no minimum point on $H_{rad,0}^1(\D)$ (see remark \ref{remark: no minimizer on disc}, also \cite{oh2021optimal}). On the other hand, $H_{rad,0}^1(\D)$ embeds naturally into $H^1_{rad}(\R^2)$ and $H$ has the corresponding extension
	\begin{equation*}
		H_{\R^2}(u):=\frac{1}{2}\int_{\R^2} |\nabla u|^2 - \frac{1}{4}\int_{\R^2} |u|^4.
	\end{equation*}
	The minima of the latter exist under the restriction $\brac{\norm{u}_{L^2(\R^2)}\leq \norm{Q}_{L^2(\R^2)}}$, and the minimal points consists of dilations and phase rotations of the ground state solution $Q$, that is, $e^{i\theta}Q_\delta$, for $\theta \in \mathbb{T}$, $\delta>0$ and $Q_\delta = \delta^{-1}Q(\delta^{-1}\cdot)$. These minimizers form a 2-dimensional submanifold, the soliton manifold\footnote{The name comes from the soliton solution of focusing nonlinear Schr\"{o}dinger equation. See, e.g. \cite{weinstein1982nonlinear, Frank2014GROUNDSO, nakanishi2011invariant}} , $\mathcal{M}$. It is natural to divide our domain $H_{rad,0}^1(\D)$ into two parts: the $\epsilon$-neighborhood in $L^2$ of the soliton manifold $U_\epsilon$ and its complement.
	
	The extremal problem for $H_{\R^2}$ is closely related to the sharp Gagliardo-Nirenberg-Sobolev (GNS) inequality (see Proposition \ref{prop: GNS inequality}). Away from the soliton manifold $\mathcal{M}$, the GNS inequality is not saturated. As a corollary,
	\begin{equation*}
		H(u) \geq c \int_\D |\nabla u|^2 dx,
	\end{equation*}
	for some $c>0$. Heuristically\footnote{Of course, the restriction $\norm{u}_{L^2(\D)}\leq \norm{Q}_{L^2(\R^2)}$ is implemented in the expectation. We omit here and below for simplicity.},
	\begin{equation*}
		\E_\mu \sqbrac{e^{\frac{1}{4}\int_{\D}|u|^{p} dx}, U_\epsilon^C} \leq \int e^{-c\int_\D |\nabla u|^2}du <\infty.
	\end{equation*}
	The integral on the right is bounded (formally) because its integrand can be regarded the density of the mean-zero Gaussian measure with covariance $\Delta^{-1}/2c$.
	
	On the neighborhood $U_\epsilon$, the idea is to expand the Hamiltonian $H$ along the soliton manifold $\mathcal{M}$. First, $U_\epsilon$ can be treated as a normal bundle of $\mathcal{M}$ (with respect to the $H^1$ inner product). More precisely, for $u \in U_\epsilon$, there is a unique decomposition $u = e^{i\theta}Q_\delta +v$ with $v$ in $V_{\theta, \delta}$, the normal vector space at $e^{i\theta}Q_\delta$. Formally,
	\begin{equation*}
		\E_\mu \sqbrac{e^{\frac{1}{4}\int_{\D}|u|^{p} dx}, U_\epsilon} = \int_{U_\epsilon} e^{-H(u)} du.
	\end{equation*}
	A change of variable formula allows us to rewrite it as a double integral
	\begin{equation}\label{eqn: double integral}
		\int_{\mathcal{M}} \int_{V_{\theta,\delta}} e^{-H(e^{i\theta}Q_\delta +v)} dv d\sigma,
	\end{equation}
	for some surface measure $\sigma \ll d\theta d\delta$. Next, we expand the Hamiltonian as
	\begin{equation}
	\label{eqn: B-}
		H(e^{i\theta}Q_\delta +v) = H(e^{i\theta}Q_\delta) + \frac{1}{2}\int_{\D}|\nabla v|^2 + B_{\theta,\delta}(v) + \mbox{higher order terms}.
	\end{equation}
	By recognizing that $e^{-\frac{1}{2}\int|\nabla v|^2}dv$ is the free Gaussian on $V_{\theta, \delta}$, denoted as $\mu_{V_{\theta, \delta}}$, the double integral (\ref{eqn: double integral}) becomes
	\begin{equation*}
		\int_{\mathcal{M}}e^{-H(e^{i\theta}Q_\delta)} \pare{\int_{V_{\theta,\delta}} e^{-B_{\theta,\delta}(v) + \text{h.o.t.}} dv} d\sigma.
	\end{equation*}
	$H(e^{i\theta}Q_\delta) \approx H_{\R^2}(Q)$ is uniformly bounded, and the high order terms can be tamed similarly as on $U_\epsilon^C$.
	
	The term $B_{\theta,\delta}(v)$ in \eqref{eqn: B-} is bounded from below by a quadratic form $\jbrac{\mathcal{A}_{\theta,\delta} v, v}_{\dot{H}^1}$ (on $H^1_{rad,0}(\D)$). The analysis of this quadratic form is more involved than the in the one-dimensional case treated in \cite{oh2021optimal}. We will show that $\mathcal{A}_{\theta,\delta}$ is compact. Expressing the free Gaussian measure $\mu_{V_{\theta,\delta}}$ in the basis consisted of $\mathcal{A}_{\theta, \delta}$'s eigenfunctions, the integral of $e^{-B_{\theta,\delta}(v)}$ term can be diagonalized and bounded by
	\[\prod_n \pare{1+2\lambda_n}^{-1/2} \lesssim \exp\pare{-\sum_n \lambda_n},\]
	with $\lambda_n$ be $\mathcal{A}_{\theta, \delta}$'s eigenvalues. Therefore, we showed $e^{-B_{\theta,\delta}(v)}$ is integrable by proving an asymptotic lower bound for $\lambda_n$.
	
	\subsection{The outline of the paper}
	In Section 2, we introduce some notations and preliminaries. In Sections 3 and 4, we bound the integral defining the partition function away from the soliton manifold. Section 5 is devoted to the normal bundle decomposition of the neighborhood $U_\epsilon$ of the soliton, and a change of variables formula adapted to this decomposition. In Section 6, we conduct some spectral analysis to bound the integral of the quadratic part $e^{-B_{\theta,\delta}(v)}$. We finish the proof of our main theorem in Section 7.
	
	\section{Notation and preliminary}
	In this section, we summarize some of the notation we will use. $A \lesssim(\gtrsim) B$ means $A\leq(\geq) CB$ for some $C>0$; $A\sim B$ means $A \lesssim B$ and $A\gtrsim B$. We denote $A \ll B$ as $A \leq cB$ for some small $c$.
	
	In this paper, all function spaces consist of complex-valued functions. More precisely, for $\Omega=\R^2$ or $\D$, $L^2(\Omega)$ is equipped with inner product
	\begin{equation*}
		\jbrac{f,g}:= \Re \int_\Omega f(x)\overline{g(x)}dx.
	\end{equation*}
	The subscript $_{rad}$ indicates the subspace of radial functions, e.g. $L^2_{rad}(\D)$. We use $H^1_{rad,0}(\D)$ to denote the subspace of radial functions in $H^1(\D)$ which vanish on $\partial \D$.
	
	\subsection{Eigenfunctions and eigenvalues}
	
	On Hilbert space, consider $-\Delta$ as an operator with domain on $H^2_{rad}(\D)$ with Dirichlet boundary conditions. Using the radial variable $r = |x|$, the eigenvalue equation can be written as
	\begin{equation*}
		\begin{bracket}
			-(\partial_r^2+\frac{1}{r}\partial_r)e &= \lambda e,\\
			e'(0)=0,\; &e(1) = 0.
		\end{bracket}
	\end{equation*}
	The $L^2$-normalized eigenfunctions and eigenvalues are
	\begin{equation}\label{def: eigenfunction}
		e_n(r) = \frac{J_0(z_n r)}{\norm{J_0(z_n \cdot)}_{L^2(\D)}}, \quad \lambda_n=z_n^2,
	\end{equation}
	where $J_0$ is the Bessel function of order $0$. $z_n$ is its $n$th zero, with asymptotic behavior (see also Lemma 2.2 in \cite{tzvetkov2006invariant}) 
	\begin{equation}\label{est: roots of J_0}
		z_n = \pi(n-\frac{1}{4})+O(\frac{1}{n}).
	\end{equation}
	By Sturm–Liouville theory, $\brac{e_n}$ (more precisely, $\brac{e_n, ie_n}$) forms an orthonormal basis for $L_{rad}^2(\D)$.
	
	With the notation above, we obtain a series representation (\ref{def: u series}) of a random element $u$ distributed according to the free Gaussian measure $\mu$ is
	\begin{equation}\label{def: series u precise}
		u(\omega) = \sum_{n\geq 1}\frac{g_n(\omega)}{z_n}e_n.
	\end{equation}
	
	We can then define the dyadic\footnote{In this paper, the capital letter $N$ (and $M$) is always to denote some dyadic number $2^n$.} projection $P_{\leq N}$ on the $2N$-dimensional (real) subspace
	\begin{equation*}
		E_N:=\text{span}\brac{e_n: n\leq N},
	\end{equation*}
	and $P_N: = P_{\leq N}- P_{\leq N/2}$.
	
	\subsection{Bessel functions of order 0}\label{section: Bessel functions}
	For the spectral analysis in Section \ref{section: spectrum analysis}, we need some asymptotic expansions for Bessel functions. The equation
	\begin{equation*}
		\partial_r^2 u + \frac{1}{r}\partial_r u +u = 0
	\end{equation*}
	has a fundamental set of solution $J_0$, $Y_0$. $J_0$ is called Bessel function of the first kind while $Y_0$ is Bessel function of second kind.
	
	When $r\rightarrow 0$, the Bessel functions have the following series expansions (see Section 10.8 in \cite{nist}):
	\begin{gather*}
		J_0(r) = \sum_{m=0}^{\infty}\frac{(-1)^m}{(m!)^2}\pare{\frac{r}{2}}^{2m},\\
		Y_0(r) = \frac{2}{\pi}\ln(\frac{r}{2})J_0(r) + \frac{2}{\pi}p(r),
	\end{gather*}
	where
	\begin{equation*}
		p(r) = \sum_{m=0}^{\infty}\frac{(-1)^m}{(m!)^2}\pare{\gamma - \sum_{k=1}^m \frac{1}{k}}\pare{\frac{r}{2}}^{2m},
	\end{equation*}
	and $\gamma$ is Euler's constant. In particular, $J_0(0) = 1$ and $Y_0(r) \approx \ln(r)$, as $r\rightarrow 0$.
	
	For $r\rightarrow \infty$, there are the following asymptotic expansions (see \cite[Section 10.7]{nist}):
	\begin{equation}\label{est: Bessel asymtotic at infity}
		\begin{aligned}
			J_0(r) = \sqrt{\frac{2}{\pi r}}\cos\pare{r-\frac{1}{4}\pi}+o(\frac{1}{r^{1/2}}),\\
			Y_0(r) = \sqrt{\frac{2}{\pi r}}\sin\pare{r-\frac{1}{4}\pi}+o(\frac{1}{r^{1/2}}).
		\end{aligned}
	\end{equation}

	\subsection{Gagliardo-Nirenberg-Sobolev inequality}
	The sharp Gagliardo-Nirenberg-Sobolev (GNS) inequality in $\R^d$, $d\geq 2$ was proved by Weinstein \cite{weinstein1982nonlinear}. We state the $d=2$ case. See e.g. \cite{Frank2014GROUNDSO} for a proof.
	\begin{proposition}\label{prop: GNS inequality}
		For any $p>2$, $u\in H^1(\R^2)$ satisfies
		\begin{equation}\label{ineq: GNS inequality}
			\norm{u}_{L^p(\R^2)}^p \leq \frac{p}{2}\norm{Q}_{L^2(\R^2)}^{2-p} \norm{\nabla u}_{L^2(\R^2)}^{p-2} \norm{u}_{L^2(\R^2)}^2
		\end{equation}
		where $Q$ is the unique positive, radial and exponentially decaying solution of
		\begin{equation}\label{eqn: ground state}
			(p-2)\Delta Q + 2Q^{p-1} - 2Q=0.
		\end{equation}
		The equality holds if and only if $u(x) = aQ(b(x-c))$ for $a\in \mathbb{C}\setminus\{0\}$, $b>0$ and $c\in\R$.
	\end{proposition}
	
	\begin{remark}\label{remark: no minimizer on disc}
		If we restrict the inequality to $H^1_{rad,0}(\D)$, then there are no functions that saturate this inequality. Indeed, if $u\in H^1_{rad,0}(\D)$ minimizes the GNS inequality on $\D$, since $u$ vanishes on $\partial \D$, it is also a minimizer of (\ref{ineq: GNS inequality}). This means $u$ equals to some $aQ(bx)$, which is positive on the whole $\R^2$, a contradiction.
	\end{remark}
	The function $Q$ is called the \emph{ground state} solution. We state some simple bounds for $Q$.
	\begin{lemma}\label{lemma: ground state}
		The radial ground state solution $Q = Q(r)$ satisfies
		\begin{enumerate}[(1)]
			\item $Q(r)$ is decreasing for $r\in [0,\infty)$.
			\item $Q(r) , \nabla Q(r)= O(e^{-cr})$.
			\item $Q(0) = \norm{Q}_{\infty} \geq (p/2)^{1/(p-2)}$.
		\end{enumerate}	
	\end{lemma}
	
	\begin{proof}
		The proof of part (1), (2) can be found, for example, in \cite[Appendix B]{tao2006nonlinear}.
		
		For part (3), as a radial function, (\ref{eqn: ground state}) can be written as
		\begin{equation*}
			Q_{rr}+ \frac{1}{r}Q_r +\frac{2}{p-2}Q^{p-1}-\frac{2}{p-2}Q = 0.
		\end{equation*}
		Define the \emph{energy}
		\begin{equation*}
			E[Q(r)] = \frac{1}{2}Q_r^2(r)+\frac{2}{p(p-2)}Q^p(r)-\frac{1}{p-2}Q^2(r).
		\end{equation*}
		A direct computation shows
		\begin{equation*}
			\frac{d}{dr}E[Q(r)] = -\frac{1}{r}Q_r^2(r)\leq 0.
		\end{equation*}
		Since $Q$ is radial, $Q_r(0)=0$. Moreover, $Q(\infty) = Q_r(\infty)=0$ implies
		\begin{equation*}
			\frac{2}{p(p-2)}Q^p(0)-\frac{1}{p-2}Q^2(0) = E[Q(0)] \geq E[Q(\infty)] = 0.
		\end{equation*}
		Since $Q>0$, this implies $Q(0)\geq (p/2)^{1/(p-2)}$. Since $Q$ is decreasing,  $\norm{Q}_{L^{\infty}(\R^2)} = Q(0)$. 
	\end{proof}

	\section{Away from the soliton manifold}\label{section: away from soliton}
	The argument for this part is in the the same spirit as the small truncation $K\ll 1$ case. We follow a modification of Bourgain's argument, as in \cite{oh2021optimal}.
	
	We begin with a large deviation result for $\mu$.
	\begin{lemma}\label{lemma: large deviation}
		For any $N \leq \infty$, 
		\begin{equation*}
			\mathbb{E}_{\mu}\sqbrac{\norm{P_{\geq N}u}_{L^4(\D)}} \lesssim (\log N)^{\frac{1}{4}} N^{-\frac{1}{2}}.
		\end{equation*}
		In particular,
		$\int_{\D}|u|^4 dx$ is $\mu$-almost surely finite.
	\end{lemma}
	\begin{proof}
		We use a bound for eigenfunctions (see \cite{tzvetkov2006invariant}):
		\begin{equation*}
			\norm{e_n}_{L^4(\D)} \lesssim \pare{\log(2+n)}^{\frac{1}{4}}.
		\end{equation*}
		Then,
		\begin{multline*}
			\mathbb{E}_{\mu}\sqbrac{\norm{P_{\geq N}u}_{L^4(\D)}^4} = \mathbb{E}\sqbrac{\int_\D 2\pare{\sum_{n\geq N} g_n^2\frac{e_n^2}{z_n^2}}^2-\sum_{n\geq N} g_n^4 \frac{e_n^4}{z_n^4}} \lesssim \sum_{\substack{n\neq m \\ n,m\geq N}} \int_\D \mathbb{E}\sqbrac{g_n^2g_m^2}\frac{e_n^2}{z_n^2}\frac{e_m^2}{z_m^2}\\
			\leq \pare{\sum_{n\geq N} \frac{1}{z_n^2}\norm{e_n}_4^2}^2 \lesssim \pare{\sum_{n\geq N} \frac{(\log n)^{1/2}}{n^2}}^2  \lesssim \log N \cdot N^{-2}.
		\end{multline*}
		The last inequality above used summation by parts:
		\begin{multline*}
			\sum_{n\geq N}\frac{(\log n)^{1/2}}{n^2} \leq \sum_{k\geq\log N} \sqrt{k}2^{-k} = 2\sqrt{\log N}N^{-1} +\sum_{k\geq \log N}\pare{\sqrt{k+1}-\sqrt{k}}2^{-k}\\
			\leq 2\sqrt{\log N}N^{-1} +\sum_{k\geq \log N} \frac{1}{\sqrt{k}}2^{-k} \leq 3(\log N)^{1/2} N^{-1}.
		\end{multline*}
	\end{proof}  
	
	Instead of directly working on the complement of some $\epsilon$-neighborhood of the soliton manifold $\mathcal{M}$, we start with an alternative characterization of the sharpness of GNS inequality. For small $\gamma>0$, define subdomain
	\begin{equation*}
		S_{\gamma}:=\bigcap\limits_N\brac{u\in L^2_{rad}(\D): \begin{aligned}
				\norm{u}_{L^2(\D)} &\leq \norm{Q}_{L^2(\R^2)},\\ \frac{1}{4}\int_{\D} |P_{\leq N} u|^4 &\leq \frac{1-\gamma}{2}\int_{\D} |\nabla P_{\leq N}u|^2.
		\end{aligned}}
	\end{equation*}
	
	Our aim in this section is the following.
	\begin{proposition}\label{prop: outside ground state}
		\begin{equation*}
			\E_\mu\sqbrac{e^{\frac{1}{4}\int_{\D} |u|^4},S_{\gamma}}<\infty
		\end{equation*}
	\end{proposition}
	
	\begin{proof}
		We further slice $S_{\gamma}$ according to its $L^4$ size:
		\begin{equation*}
			F_N := \left\{ u\in L_{rad}^2(\D): \norm{P_{\geq M}u}_4 > (\log M)^{3/4}, \, M< N; \; \norm{P_{\geq N}u}_4 \leq (\log N)^{3/4} \right\}
		\end{equation*}
		As a corollary of Fernique's theorem \cite{fernique1975regularite} (see also Lemma 4.2 in \cite{oh2021optimal}), there exists some $c>0$, such that
		\begin{equation*}
			\mu\pare{\norm{P_{\geq N}u}_4 \geq t \mathbb{E}\sqbrac{\norm{P_{\geq N}u}_4}} \leq e^{-ct^2}.
		\end{equation*}
		By Lemma \ref{lemma: large deviation}, taking $t = (\log N)^{3/4} \cdot \mathbb{E}\sqbrac{\norm{P_{\geq N}u}_4}^{-1} \gtrsim (\log N)^{1/2}N^{1/2}$,
		\begin{align}\label{ineq:section2 F_N}
			\mu(F_{2N}) \leq \mu(\norm{P_{\geq N}u}_4 \geq (\log N)^{3/4}) \leq  \exp\pare{-c(\log N)N}.
		\end{align}
		By Young's inequality,
		\begin{equation}\label{ineq:section2 holder}
			\int_\D |u|^4 = \int_\D |P_{\leq N}u + P_{\geq 2N}u|^4 \leq (1+\epsilon)\int_\D |P_{\leq N}u|^4 + C_{\epsilon}\int_\D |P_{\geq 2N}u|^4.
		\end{equation}
		On $S_{\gamma}\cap F_{2N}$, using  \eqref{ineq:section2 holder}, the definition of $S_{\gamma}$ and H\"{o}lder's inequality,
		\begin{align*}
			\mathbb{E}&\sqbrac{\exp\pare{\frac{1}{4}\int_\D |u|^4},
				S_{\gamma}\cap F_{2N}}\leq e^{\frac{C_{\epsilon}}{4}(\log N)^3}\mathbb{E}\sqbrac{\exp\pare{\frac{1}{4}\int_\D (1+\epsilon)|P_{\leq N}u|^4},S_{\gamma}\cap F_{2N}}\\
			&\leq e^{\frac{C_{\epsilon}}{4}(\log N)^3}\mathbb{E}_{\mu}\sqbrac{\exp\pare{\frac{(1-\gamma)(1+\epsilon)}{2}\int_\D|\nabla P_{\leq N} u|^2}, F_{2N}}\\
			&\leq e^{\frac{C_{\epsilon}}{4}(\log N)^3} \mathbb{E}_{\mu}\sqbrac{\exp\pare{\frac{(1-\gamma)(1+\epsilon)(1+\eta)}{2}\int_\D|\nabla P_{\leq N} u|^2}}^{\frac{1}{1+\eta}}\mu(F_{{2N}})^{\frac{\eta}{1+\eta}}\\
			&\leq e^{C_1(\log N)^3 - c_2(\log N)N}\E_{\mu}\sqbrac{\exp\pare{\frac{c_3}{2}\int_\D |\nabla P_{\leq N}u|^2}}^{\frac{1}{1+\eta}}.
		\end{align*}
		Above $C_1 = C_1(\epsilon)$, $c_2 = c_2(\eta)$, and $c_3 = (1-\gamma)(1+\epsilon)(1+\eta)<1$ by choosing $\epsilon, \eta$ small enough. Since the expectation in the last line only involves $P_{\leq N}u\in E_N$,
		\begin{equation}\label{eqn: exponential gaussian}
			\E_{\mu}\sqbrac{\exp\pare{\frac{c_3}{2}\int_\D |\nabla P_{\leq N}u|^2}} = \prod_{n\leq N}\E_{\mathbb{P}}\sqbrac{e^{\frac{c_3}{2}g_n^2}} = \exp\pare{\frac{N}{2}\log(\frac{1}{1-c_3})}.
		\end{equation}
		Summing over $N$, $\mathbb{E}_\mu\sqbrac{\exp\pare{\frac{1}{4}\int_B |u|^4}, S_{\gamma}}$ is bounded.
	\end{proof}

	\section{Reduce to the soliton neighborhood: a stability argument}\label{section: reduce to the soliton neighborhood}
	The sharp GNS inequality (\ref{ineq: GNS inequality}), at critical $p=4$ for $u\in H^1_{rad}(\R^2)$, reads as
	\begin{equation}\label{ineq: GNS Hamiltonian}
		\frac{1}{4}\norm{u}_{L^4(\R^2)}^4 \leq \frac{1}{2}\norm{\nabla u}_{L^2(\R^2)}^2 \pare{\frac{\norm{u}_{L^2(\R^2)}}{\norm{Q}_{L^2(\R^2)}}}^2.
	\end{equation}
	On $L^2$-cutoff $\norm{u}_{L^2(\R^2)}\leq \norm{Q}_{L^2(\R^2)}$, this implies
	\begin{equation*}
		H_{\R^2}(u) \geq 0.
	\end{equation*}
	The minima of this functional coincide with the minimizers of GNS inequality with restriction $\norm{u}_{L^2(\R^2)} = \norm{Q}_{L^2(\R^2)}$. By Proposition \ref{prop: GNS inequality}, these minimum points form a 2-dimensional (real) submanifold, the soliton manifold
	\begin{equation*}
		\mathcal{M}_{\R^2} := \brac{e^{i\theta}Q_{\delta}: \theta \in \mathbb{T}, \delta>0},
	\end{equation*}
	where
	\begin{equation*}
		Q_{\delta}(x) = \delta^{-1}Q(\delta^{-1}x).
	\end{equation*}
	The scaling is chosen to keep the $L^2(\R^2)$ norm invariant. By (\ref{eqn: ground state}),
	\begin{equation}\label{eqn: scaling ground state}
		\Delta Q_\delta + Q_\delta^3 - \delta^{-2}Q_\delta = 0.
	\end{equation}
	Since $H^1_{rad,0}(\D)$ embeds naturally into $H^1(\R^2)$, \eqref{ineq: GNS Hamiltonian} holds true on $\D$. In particular, $H(u)> 0$ when $\norm{u}_{L^2(\D)} \leq \norm{Q}_{L^2(\R^2)}$.
	
	The $(\delta_*, \delta^*)$-segment of the $\epsilon$-neighborhood of $\mathcal{M}_{\R^2}$ in $L^2_{rad}(\D)$ is defined as
	\begin{equation*}
		U_{\epsilon}(\delta_*, \delta^*):=\brac{ u\in L_{rad}^2(\D):
			\begin{aligned}
				\norm{u}_{L^2(\D)} &\leq \norm{Q}_{L^2(\R^2)},\\
				\norm{u-e^{i\theta}Q_{\delta}}_{L^2(D)}\leq \epsilon, &\; \text{for some } \theta\in \mathbb{T},\, \delta \in  (\delta_*,\delta^*)
		\end{aligned}}.
	\end{equation*}
	
	In Section \ref{section: away from soliton}, we proved integrability on $S_\gamma$. The following stability result shows that $S_\gamma$ contains $U_\epsilon^C$.
	
	\begin{lemma}\label{lemma: stability}
		Given $\epsilon$, $\delta^*$,  there exists $\gamma = \gamma(\epsilon, \delta^*)$, such that
		\begin{equation*}
			U_{\epsilon}(0,\delta^*)^C\subset S_{\gamma},
		\end{equation*}
		where the complement of $U_\epsilon$ is taken in
		\[\brac{u\in L^2_{rad}(\D): \norm{u}_{L^2(\D)} \leq \norm{Q}_{L^2(\R^2)}}.\]
	\end{lemma}
	\begin{proof}
		The proof follows the idea of Lemma 6.3 in \cite{oh2021optimal}.
		
		Suppose by contradiction that there exists $\epsilon>0$ and $u_n \notin U_{\epsilon}$, $\norm{u_n}_{L^2(\D)}\leq \norm{Q}_{L^2(\R^2)}$, such that, for some $N_n>0$ and $\gamma_n\rightarrow 0$, 
		\begin{equation}\label{eqn: section2 unstable}
			\frac{1}{4}\int_\D |P_{\leq N_n} u|^4 > \frac{1-\gamma_n}{2}\int_\D |\nabla P_{\leq N_n}u|^2.
		\end{equation}
		By GNS inequality (\ref{ineq: GNS inequality}), on the other hand,
		\begin{equation*}
			\frac{1}{4}\int_\D |P_{\leq N_n} u|^4 \leq \frac{1}{2}\int_\D |\nabla P_{\leq N_n}u|^2 \frac{\norm{P_{\leq N_n}u_n}_{L^2(\D)}^2}{\norm{Q}_{L^2(\R^2)}^2}.
		\end{equation*}
		By (\ref{eqn: section2 unstable}),
		\begin{equation*}
			1-\gamma_n \leq \frac{\norm{P_{\leq N_n}u_n}_{L^2(\D)}^2}{\norm{Q}_{L^2(\R^2)}^2} \leq \frac{\norm{u_n}_{L^2(\D)}^2}{\norm{Q}_{L^2(\R^2)}^2}\leq 1.
		\end{equation*}
		Thus
		\begin{align}\label{est: section2 u_n}
			\norm{P_{\leq N_n}u_n}_{L^2(\R^2)} \rightarrow \norm{Q}_{L^2(\R^2)}^2, \quad \norm{u_n - P_{\leq N_n}u_n}_{L^2(\R^2)}\rightarrow 0.
		\end{align}
		
		If $\norm{P_{\leq N_n}u_n}_{\dot{H}^1(\D)}$ is uniformly bounded, then $P_{\leq N_n}u_n$ converges weakly to some $v \in H^1(\D)$. Definition of the eigenfunctions (\ref{def: eigenfunction}) implies $P_{\leq N_n}u_n = 0$ on $\partial \D$, hence in $H^1(2\D)$.  By Rellich-Kondrachov theorem, $P_{\leq N_n}u_n \rightarrow v$ in $L^2\cap L^4(2\D)$ and $v=0$ on $\D^c$, which in turn means $v \in H^1(\R^2)$. Using (\ref{eqn: section2 unstable}) again, we have
		\begin{equation*}
			\frac{1}{4}\int_\D |v|^4 = \lim_{n\rightarrow \infty} \frac{1}{4}\int_\D |P_{\leq N_n} u|^4 \geq \liminf_{n\rightarrow \infty} \frac{1}{2}\int_\D |\nabla P_{\leq N_n}u|^2  \geq \frac{1}{2}\int_\D |\nabla v|^2.
		\end{equation*}
		This shows $v \in H^1(\R^2)$ is an optimizer of GNS inequality supported on $\D$, a contradiction. Therefore, up to a  subsequence,
		\begin{equation*}
			d_n := \norm{P_{\leq N_n}u_n}_{\dot{H}^1(\D)} \rightarrow \infty.
		\end{equation*}
		
		Set $w_n = d_n^{-1}P_{\leq N_n}u_n(d_n^{-1}\cdot)$. By scaling,
		\begin{equation*}
			\norm{w_n}_{\dot{H}^1(\R^2)} = 1, \text{ and } \norm{w_n}_{L^2(\R^2)}\rightarrow \norm{Q}_{L^2(\R^2)}.
		\end{equation*}
		and by (\ref{eqn: section2 unstable}),
		\begin{equation}\label{eqn: section2 w_n L4}
			\limsup_{n\rightarrow \infty}\frac{1}{4}\int_{\R^2} |w_n|^4 \geq \frac{1}{2}.
		\end{equation}
		
		Now, using the bubble decomposition (see \cite{hmidi2005blowup}, Proposition 3.1), there exist $J^*\leq \infty$ functions $\phi_j \in H^1(\R^2)$, and $x_n^j \in \R^2$ and for $J\leq J^*$, there exist $r_n^J \in H^1(\R^2)$, such that
		\begin{equation*}
			w_n(x) = \sum_{j=1}^J\phi_j(x-x_n^j) + r_n^J(x)
		\end{equation*}
		with the following properties:
		\begin{gather*}
			\norm{w_n}_{L^2}^2 = \sum_{j=1}^J \norm{\phi_j}_{L^2}^2 + \norm{r_n^J}_{L^2}^2 + o(1),\\
			\norm{\nabla w_n}_{L^2}^2 + \sum_{j=1}^J \norm{\nabla \phi_j}_{L^2}^2 = \norm{\nabla r_n^J}_{L^2}^2 + o(1),\\
			\limsup_{n\rightarrow \infty}\norm{w_n}_{L^4}^4 = \sum_{j=1}^{J^*}\norm{\phi_j}_{L^4}^4,\\
			\limsup_{J\rightarrow J^*}\limsup_{n\rightarrow \infty}\norm{r_n^J}_{L^4}^4 =0.
		\end{gather*}
		If $J^* = 0$, then $\norm{w_n}_{L^4} = \norm{r_n^0}_{L^4} \rightarrow 0$, which contradicts (\ref{eqn: section2 w_n L4}). For $J^*\geq 1$, combining with (\ref{eqn: section2 w_n L4}),
		\begin{multline*}
			\frac{1}{2}\leq \limsup_{n\rightarrow \infty}\frac{1}{4}\int_{\R^2} |w_n|^4  = \frac{1}{4}\sum_{j=1}^{J^*}\norm{\phi_j}_{L^4}^4 \overset{GNS}{\leq} \frac{1}{2} \sum_{j=1}^{J^*}\norm{\nabla \phi_j}_{L^2}^2 \frac{\norm{\phi_j}_{L^2}^2}{\norm{Q}_{L^2}^2}\\
			\leq \sup_j \norm{\phi_j}_{L^2}^2 \frac{1}{2}\frac{\limsup_{n\rightarrow \infty}\norm{\nabla w_n}_{L^2}^2}{\norm{Q}_{L^2}^2} =\frac{1}{2} \frac{\sup_j \norm{\phi_j}_{L^2}^2}{\norm{Q}_{L^2}^2}.
		\end{multline*}
		On the other hand, 
		\begin{equation*}
			\frac{1}{2} \frac{\sup_j \norm{\phi_j}_{L^2}^2}{\norm{Q}_{L^2}^2} \leq \frac{1}{2}\sum_{j=1}^{J^*}\frac{\norm{\phi_j}_{L^2}^2}{\norm{Q}_{L^2}^2}= \frac{1}{2}\frac{\limsup_{n\rightarrow \infty}\norm{w_n}_{L^2}^2}{\norm{Q}_{L^2}^2} = \frac{1}{2}.
		\end{equation*}
		Thus $J^* = 1$, and $w_n(x) = \phi_1(x-x_n^1) + r_n^1(x)$. This implies that $\phi_1$ is an optimizer of GNS inequality. Moreover, since $\norm{r_n^1}_{L^4}\rightarrow 0$, up to subsequence $w_n \rightarrow \phi_1(\cdot -x_n^1)$ a.e.; since $\norm{w_n}_{L^2} \rightarrow \norm{\phi_1}_{L^2}$, $w_n - \phi_1(\cdot -x_n^1) \rightarrow 0$ in $L^2$. Since $w_n$ is radial, $x_n^1$ has to be bounded. We may assume $x_n^1 \rightarrow x_0$. By uniqueness of the optimizer, $\phi_1(\cdot - x_0) = e^{i\theta}Q_{\delta}$, for some $\theta, \delta$. By the definition of $w_n$ and (\ref{est: section2 u_n}),
		\begin{equation*}
			\lim_{n\rightarrow \infty}\norm{u_n - e^{i\theta}Q_{d_n^{-1}\delta}}_{L^2} = \lim_{n\rightarrow \infty}\norm{P_{\leq N_n}u_n - e^{i\theta}Q_{d_n^{-1}\delta}}_{L^2} = 0,
		\end{equation*}
		contradicting to $u_n \notin U_{\epsilon}$.
	\end{proof}
	
	As a corollary of Proposition \ref{prop: outside ground state} and Lemma \ref{lemma: stability}, we have
	\begin{corollary}\label{cor: integrability near soliton manifold}
		For $\epsilon, \delta^*>0$,
		\begin{equation*}
			\E_\mu\sqbrac{e^{\frac{1}{4}\int_{\D} |u|^4}, U_{\epsilon}(0,\delta^*) ^C}<\infty.
		\end{equation*}
	\end{corollary}

	\section{Normal bundle decomposition} 
	We now move on to the integrability on the soliton neighborhood $U_\epsilon$. The goal for the remainder of the paper is to bound
	\begin{equation}\label{term: expectation on soliton neighborhood}
		\E_\mu\sqbrac{e^{\frac{1}{4}\int_{\D} |u|^4}, U_{\epsilon}(0,\delta^*)}.
	\end{equation}
	Recall the Gibbs measure $\rho$ is formally written as $e^{-H} du$. As discussed in Section \ref{section: reduce to the soliton neighborhood}, the minima of the Hamiltonian $H_{\R^2}^1$ are achieved on the soliton manifold $\mathcal{M}_{\R^2}$, the functions of form $e^{i\theta}Q_\delta$. As the scaling parameter $\delta$ tends to zero, $e^{i\theta}Q_\delta$ will essentially concentrate inside $\D$. Therefore, the restrictions of $e^{i\theta}Q_\delta$ to $\D$ almost minimize $H$, and form a ``near soliton" manifold $\mathcal{M}$ in $U_\epsilon$. The restriction of $e^{i\theta}Q_\delta$ does not lie in $H_{rad,0}(\D)$. For the spectral analysis in Section \ref{section: quadratic part}, we define the restriction, written in boldface, as\footnote{Here we use radial variable $r = |x|$.}
	\begin{equation}\label{est: restriction of Q}
		\bQ_\delta(r):=Q_\delta(r)-Q_\delta(1) = Q_\delta(r) + O(e^{-c\delta^{-1}}),
	\end{equation}
	and
	\begin{equation*}
		\mathcal{M}:=\brac{e^{i\theta}\bQ_\delta: \theta\in \mathbb{T}, \delta>0}.
	\end{equation*}
	
	To estimate $H$ on $U_\epsilon$, it is reasonable to expand it along normal directions to $\mathcal{M}$.  By this we mean that at each point $e^{i\theta}\bQ_\delta \in \mathcal{M}$, we perturb $H$ along the normal vector space
	\begin{equation}\label{def: normal space}
		V_{\theta,\delta}:=\brac{v\in L^2_{rad}(\D): \jbrac{v, -\Delta (e^{i\theta}\partial_{\delta}\bQ_{\delta})} = \jbrac{v, -\Delta (ie^{i\theta}\bQ_{\delta})} = 0}.
	\end{equation}
	Note that
	\begin{align*}
		\partial_{\theta}\pare{e^{i\theta}\bQ_\delta} &= ie^{i\theta}\bQ_{\delta},\\
		\partial_{\delta}\pare{ie^{i\theta}\bQ_{\delta}} = e^{i\theta}\partial_{\delta}\bQ_{\delta} &= e^{i\theta}\pare{\partial_{\delta}Q_{\delta} - \partial_{\delta}Q_{\delta}(1)}.
	\end{align*}
	$ie^{i\theta}\bQ_{\delta}$ and $e^{i\theta}\partial_{\delta}\bQ_{\delta}$ are two tangent vectors of $\mathcal{M}$ in $H_{rad,0}^1(\D)$. Since our inner product is real-valued,
	\begin{equation*}
		\jbrac{ie^{i\theta}\bQ_{\delta}, e^{i\theta}\partial_{\delta}\bQ_{\delta}} = \jbrac{ie^{i\theta}\bQ_{\delta}, (-\Delta)e^{i\theta}\partial_{\delta}\bQ_{\delta}}=0.
	\end{equation*}
	
	For simplicity, we only discuss the case when $\theta = 0$ (which is sufficient for our proof). Let $u=\bQ_\delta + v$ with $v\in V_{0,\delta}$,
	\begin{align*}
		H(u) = H(\bQ_\delta + v)  = H(\bQ_\delta) &+ \jbrac{-\Delta \bQ_\delta - \bQ_\delta^3, v}+ \frac{1}{2}\int_\D |\nabla v|^2 \\
		&-\jbrac{\bQ_\delta^2, \frac{1}{2}v^2+|v|^2} - \jbrac{\bQ_\delta, |v|^2v} - \frac{1}{4}\int_\D |v|^4
	\end{align*}
	
	\noindent\textbf{Constant part:} By (\ref{est: restriction of Q}), exponential decay of the ground state $Q$ and $\nabla Q$(see Lemma \ref{lemma: ground state}),
	\begin{align*}
		H(\bQ_{\delta})
		&=\frac{1}{2}\int_\D|\nabla \bQ_{\delta}|^2-\frac{1}{4}\int_\D|\bQ_{\delta}|^4\\
		&=\frac{1}{2}\int_\D|\nabla Q_{\delta}|^2-\frac{1}{4}\int_\D|Q_{\delta}|^4 + O(e^{-c\delta^{-1}})\\
		&=\delta^{-2}\pare{\frac{1}{2}\int_{|x|>\delta^{-1}}|\nabla Q|^2 - \frac{1}{4}\int_{|x|>\delta^{-1}}|Q|^4}+O(e^{-c\delta^{-1}})= O(e^{-c\delta^{-1}}).
	\end{align*}
	
	\noindent\textbf{Linear part:}  Assume $\norm{v}_{L^2(\D)}\leq 1$, which is indeed the case on $U_\epsilon$(see Remark \ref{remark: L^2 norm of v}). By (\ref{eqn: scaling ground state}),
	\begin{equation*}
		\jbrac{-\Delta \bQ_{\delta} - \bQ_{\delta}^3,v} = \jbrac{-\Delta Q_{\delta} -Q_{\delta}^3,v} + O(e^{-c\delta^{-1}}) 
		= -\delta^{-2}\jbrac{Q_{\delta},v}+O(e^{-c\delta^{-1}}).
	\end{equation*}
	
	\begin{remark}\label{remark: lagrange multiplier}
		The linear functional above is related to $dH$, the linearization of $H$ on $H^1(\mathbb{R}^2)$. Denote $M_{\R^2}(u):=\frac{1}{2}\norm{u}_{L^2(\R^2)}^2$. Since $Q_\delta$ is a minimizer of $H_{\R^2}$ conditioned on $M_{\R^2} = \frac{1}{2}\norm{Q}_{L^2(\R^2)}^2$, there exists a Lagrange multiplier $\lambda = \lambda(Q_\delta)$ such that
		\[dH_{\R^2}(Q_\delta) - \lambda dM_{\R^2}(Q_\delta) = 0,\; \mbox{that is, } -\Delta Q_{\delta} -Q_{\delta}^3 - \lambda Q_\delta =0.\]
		Using the equation of the ground state (\ref{eqn: scaling ground state}), we know $\lambda(Q_\delta)  = -\delta^{-2}$.
	\end{remark}
	
	\noindent\textbf{Higher order terms:} Applying Cauchy-Schwartz, for any small $\eta$
	\begin{align*}
		\jbrac{\bQ_{\delta},|v|^2v}= \jbrac{Q_{\delta},|v|^2v} + O(e^{-c\delta^{-1}}) \leq \eta\jbrac{Q_{\delta},|v|^2} +\tilde{C}_{\eta}\int_\D|v|^4 +O(e^{-c\delta^{-1}}).
	\end{align*}
	
	Collecting all the reductions above, we get
	\begin{equation}\label{bound: hamiltonian lower bound}
		H(u) \geq O(e^{-c\delta^{-1}}) + \frac{1}{2}\int_\D |\nabla v|^2 - B_\delta(v) - C_{\eta}\int_\D|v|^4,
	\end{equation}
	where
	\begin{equation*}
		B_\delta(v) = \delta^{-2}\jbrac{\bQ_{\delta},v} + \jbrac{\bQ_\delta^2, \frac{1}{2}v^2+(1+\eta)|v|^2}.
	\end{equation*}
	
	\subsection{Change of variable formula}
	The decomposition $u = e^{i\theta}\bQ_\delta + v$ induces assign a normal bundle structure to the soliton neighborhood $U_\epsilon$.
	
	\begin{theorem}[Normal bundle decomposition]\label{thm:decomposition}
		Given $\epsilon>0$, and dyadic number $N\in [1,\infty]$, there exists $\delta_* = \delta_*(N,\epsilon)$, and $\delta^* = \delta^*(\epsilon)$, satisfying $N^{-1} \lesssim \delta_* < \delta^* \ll 1$. Such that at any point $e^{i\bar{\theta}}\bQ_{\bar{\delta}} \in \mathcal{M}$ with $\bar{\delta} \in (\delta_*,\delta^*)$, there are a neighborhood $W$ of $(\bar{\theta}, \bar{\delta}, 0)$ in $\mathbb{T}\times \R \times V_{\bar{\theta}, \bar{\delta}}\cap E_N$, and a diffeomorphism $G: W \rightarrow G(W) \subset L^2_{rad}(\D)\cap E_N$ defined as
		\begin{align*}
			G(\theta, \delta, v) := P_{\leq N}\pare{e^{i(\bar{\theta}+\theta)}\bQ_{\bar{\delta}\delta}}+P_{V_{\bar{\theta}+\theta,\bar{\delta}\delta}\cap E_N}v.
		\end{align*}
		Its image $G(W)$ contains $\brac{u \in L^2_{rad}(\D)\cap E_N: \norm{u-P_{\leq N}e^{i\bar{\theta}}\bQ_{\bar{\delta}}}_{L^2(\D)} \leq 2\epsilon}$.
		
		In particular, for any $u \in U_\epsilon(\delta_*,\delta^*)$,  there exist $\theta$, $\delta$ and $v\in V_{\theta,\delta}\cap E_N$, such that
		\begin{align*}
			P_{\leq N}u = P_{\leq N}\pare{e^{i\theta}\bQ_{\delta}}+v.
		\end{align*} 
	\end{theorem}
	We postpone the proof to Section \ref{section: normal bundle decomposition}.
	\begin{remark}
		Although the decomposition happens in $L^2_{rad}(\D)$, the normal vector space $V_{\theta,\delta}$ in (\ref{def: normal space}) is actually with respect to $\dot{H}^1$ inner product. We make this choice because the free Gaussian measure $\mu$ is homogeneous on its Cameron-Martin space $H^1_{rad,0}(\D)$: for any vector space $\dot{H}^1$-decomposition $H^1_{rad,0}(\D) = V_1 \oplus_{\dot{H}^1} V_2$, we have the splitting $\mu = \mu_{V_1} \otimes \mu_{V_2}$, where $\mu_{V_i}$ are the free Gaussians on $V_i$.
	\end{remark}
	
	For integration on a normal bundle, we refer the following change of variable formula.
	\begin{lemma}[Change of variable, Lemma 6.11 in \cite{oh2021optimal}]\label{lemma: change of variable}
		Let $\mathcal{M}^d\subset \R^n$ be a closed submanifold and $\mathcal{N}$ be its normal bundle. Suppose there is a decomposition (a diffeomorphism) mapping a neighborhood $U$ of $\mathcal{M}$ to $\mathcal{N}$ via
		\begin{align*}
			u = x+v, \; u\in U,\, x\in \mathcal{M},\, v\in T_x^{\perp}\mathcal{M}.
		\end{align*}
		Let $V = \{(x,v): x+v\in U\}$, then for any measurable function $f: \R^n \rightarrow \R$,
		\begin{equation*}
			\int_U f(u)du \lesssim \int_{\mathcal{M}}\int_{T_x^{\perp}\mathcal{M}}f(x+v)\mathbbm{1}_{V}(x,v)dvd\sigma(x)
		\end{equation*}
		the measure $\sigma(x)$ is defined as
		\begin{equation}\label{eqn:section3 d_sigma}
			d\sigma(x):=\pare{1+\sup_{k=1,\cdots,d}|\nabla_x t_k(x)|^d}d\omega(x),
		\end{equation}
		where $d\omega$ is the surface measure on $\mathcal{M}$ and $\{t_k(x)\}_{k=1}^d$ is an orthonormal frame of $\mathcal{M}$.
	\end{lemma}
	
	Combing with the lower bound (\ref{bound: hamiltonian lower bound}), Lemma \ref{lemma: change of variable} formally bounds the $\mu$-expectation (\ref{term: expectation on soliton neighborhood}) as
	\begin{equation*}
		\int_{\mathcal{M}}\int_{V_{\theta,\delta}}e^{B_\delta(v)+C_\eta\int_\D|v|^4}e^{-\frac{1}{2}\int_\D |\nabla v|^2}dvd\sigma(\theta,\delta).
	\end{equation*}
	Note that ``$e^{-\frac{1}{2}\int_\D |\nabla v|^2}dv = d\mu_{V_{\theta,\delta}}$". $\mu_{V_{\theta,\delta}}$, the free Gaussian measure on $V_{\theta,\delta}$ can be defined as following: Let $\brac{h_n}$ be a $\dot{H}^1$ orthonormal basis in the dense subspace $V_{\theta,\delta}\cap H^1_{rad,0}(\D)$, then $\mu_{V_{\theta,\delta}}$ is the law of random variable $v(\omega) = \sum_n g_n(\omega)h_n$, where $\{g_n\}$ is a sequence of independent standard complex-valued Gaussians. This construction coincides with $\mu$, where $\brac{\frac{e_n}{z_n}}$ is the orthonormal basis in $H^1_{rad,0}(\D)$. With this argument in mind, we expect the following estimate.
	
	\begin{proposition}\label{prop: change of variable}
		Given $\delta^*>0$,
		\begin{equation*}
			\E_\mu\sqbrac{e^{\frac{1}{4}\int_{\D} |u|^4}, U_{\epsilon}(0,\delta^*)} \lesssim \int_0^{\delta^*}\pare{\int_{V_{0,\delta}}e^{B_\delta(v)+C_\eta\int_\D|v|^4}d\mu_{V_{0,\delta}}(v)}\delta^{-5}d\delta.
		\end{equation*}
		where $\mu_{V_{0,\delta}}$ is the free Gaussian measure on $V_{0,\delta}$.
	\end{proposition}
	\begin{proof}
		To simplify the notation, we denote $U_\epsilon$ for $U_{\epsilon}(0,\delta^*)$. Up to some mollification of the indicator function $\mathbbm{1}_{U_{\epsilon}}$, the Dominated Convergence Theorem gives
		\begin{align*}
			\E_\mu\sqbrac{e^{\frac{1}{4}\int_{\D} |u|^4}, U_{\epsilon}(0,\delta^*)} = \lim_{N\rightarrow \infty} \int \mathbbm{1}_{U_{\epsilon}}(P_{\leq N}u)e^{\frac{1}{4}\int_\D|P_{\leq N}u|^4}d\mu(u).
		\end{align*}
		Since $\mu$ is the law of (\ref{def: series u precise}),
		\begin{align*}
			\int \mathbbm{1}_{U_{\epsilon}}(P_{\leq N}u)e^{\frac{1}{4}\int_\D|P_{\leq N}u|^4}d\mu(u) = \int_{E_N\cap U_\epsilon} e^{\frac{1}{4}\int_\D|\sum_{n\leq N} \frac{g_n}{z_n}e_n|^4}\prod_{n\leq N}\frac{1}{2\pi}e^{-\frac{1}{2}g_n^2}dg_n.
		\end{align*}
		Using decomposition Theorem \ref{thm:decomposition} and change of variable Lemma \ref{lemma: change of variable}, the integral above is bounded by\footnote{To simplify the notation, we hide the restriction $\delta\geq \delta_*\gtrsim N^{-1}$ for $P_{\leq N}\mathcal{M}$, enforced by Theorem \ref{thm:decomposition}.}.
		\begin{align*}
			\int_{P_{\leq N}\mathcal{M}}\int_{V_{\theta,\delta}\cap E_N} (2\pi)^{-\frac{N}{2}}e^{-H(P_{\leq N}e^{i\theta}\bQ_{\delta}+v)}dv d\sigma(\theta,\delta).
		\end{align*}
		Note that on $V_{\theta,\delta}\cap E_N$, $d\mu_{V_{\theta,\delta}} = (2\pi)^{-\frac{N-2}{2}}e^{-\frac{1}{2}\int_\D |P_{\leq N}v|^2} dv$. The integral above can be written as
		\begin{align*}
			\int_{P_{\leq N}\mathcal{M}}\int_{V_{\theta,\delta}}f(P_{\leq N}e^{i\theta}\bQ_{\delta},P_{\leq N}v)d\mu_{V_{\theta,\delta}}(v)d\sigma(\theta,\delta).
		\end{align*}
		where
		\begin{align*}
			f(P_{\leq N}e^{i\theta}\bQ_{\delta} &,P_{\leq N}v) \\
			&= (2\pi)^{-1} \exp\pare{-H\pare{P_{\leq N}(e^{i\theta}\bQ_{\delta}+v)} + \frac{1}{2}\int_\D |P_{\leq N}v|^2 }\\
			&=(2\pi)^{-1} \exp\pare{\frac{1}{4}\int_{\D}|P_{\leq N}(e^{i\theta}\bQ_{\delta}+v)|^4-\frac{1}{2}\int_{\D}|\nabla P_{\leq N}e^{i\theta}\bQ_{\delta}|^2+\jbrac{\Delta P_{\leq N}e^{i\theta}\bQ_{\delta}, v}}.
		\end{align*}	
		
		By (\ref{eqn:section3 d_sigma}), 
		\begin{equation*}
			d\sigma(\theta,\delta) = \pare{1+\sup_{k=1,2}\norm{\partial_{\theta}t_k}_{\dot{H}^1}^2+\norm{\partial_{\delta}t_k}_{\dot{H}^1}^2}d\omega(\theta,\delta).
		\end{equation*}
		Here $t_1$ is $\dot{H}^1-$normalized $P_{\leq N}ie^{i\theta}\bQ_{\delta}$ whereas $t_2$ is $\dot{H}^1-$normalized $P_{\leq N}e^{i\theta}\partial_{\delta}\bQ_{\delta}$. By scaling argument, as $\delta \rightarrow 0$,
		\begin{equation}\label{est: H^k norm of Q}
			\begin{aligned}
				\norm{Q_{\delta}}_{H^k(\D)} &= O(\delta^{-k}),\\
				\norm{\partial_{\delta}Q_{\delta}}_{H^k(\D)} &= O(\delta^{-k-1}).
			\end{aligned}
		\end{equation}
		Therefore,
		\begin{align*}
			\norm{\partial_{\theta}t_k}_{\dot{H}^1}^2=O(1), \; \norm{\partial_{\delta}t_k}_{\dot{H}^1}^2=O(\delta^{-2}).
		\end{align*}
		For the surface measure, note that vectors $\frac{\partial}{\partial \theta}, \frac{\partial}{\partial \delta}$ are orthogonal in $\dot{H}^1$. Hence
		\begin{align*}
			d\omega(\theta,\delta) = \norm{P_{\leq N}ie^{i\theta}\bQ_{\delta}}_{\dot{H}^1} \norm{P_{\leq N}e^{i\theta}\partial_{\delta}\bQ_{\delta}}_{\dot{H}^1} \lesssim  \delta^{-3} d\delta d\theta.
		\end{align*}
		Taking $N\rightarrow \infty$, the Dominated Convergence Theorem yields
		\begin{equation*}
			\int_{\mathbb{T}}\int_{\delta_*}^{\delta^*}\int_{V_{\theta,\delta}}f(P_{\leq N}e^{i\theta}\bQ_{\delta},P_{\leq N}v)d\mu_{V_{\theta,\delta}}(v)\delta^{-3} d\delta d\theta \rightarrow \int_{\mathbb{T}}\int_0^{\delta^*}\int_{V_{\theta,\delta}}f(e^{i\theta}\bQ_{\delta}, v)d\mu_{V_{\theta,\delta}}(v)\delta^{-3} d\delta d\theta.
		\end{equation*}
		
		Observe that $f(e^{i\theta}\bQ_{\delta}, v) = f(\bQ_{\delta}, e^{-i\theta}v)$ and $e^{-i\theta}V_{\theta, \delta} = V_{0, \delta}$. Therefore,
		\begin{equation*}
			\int_{\mathbb{T}}\int_0^{\delta^*}\int_{V_{\theta,\delta}}f(e^{i\theta}\bQ_{\delta}, v)d\mu_{V_{\theta,\delta}}(v)\delta^{-3} d\delta d\theta = 2\pi \int_0^{\delta^*}\int_{V_{0,\delta}}f(\bQ_{\delta}, v)d\mu_{V_{0,\delta}}(v)\delta^{-3} d\delta.
		\end{equation*}
		
		Using (\ref{bound: hamiltonian lower bound}),
		\begin{equation*}
			f(\bQ_{\delta}, v) = (2\pi)^{-1} e^{-H\pare{\bQ_{\delta}+v} + \frac{1}{2}\int_\D v|^2 } \leq (2\pi)^{-1}e^{O(e^{-c\delta^{-1}}) + B_\delta(v) + C_{\eta}\int_\D|v|^4}.
		\end{equation*}
		This gives the desired bound.
	\end{proof}
	\begin{remark}
		If we keep the constraint $U_\epsilon$ in the proof, the bound in Proposition \ref{prop: change of variable} can be strengthened to
		\begin{equation*}
			\int_0^{\delta^*}\pare{\int_{V_{0,\delta}}\mathbbm{1}_{U_\epsilon}(\bQ_\delta+v) e^{B_\delta(v)+C_\eta\int_\D|v|^4}d\mu_{V_{0,\delta}}(v)}\delta^{-5}d\delta.
		\end{equation*}
		In particular, by Theorem \ref{thm:decomposition}, this gives the bound
		\begin{equation}\label{est: strengthened change of variable}
			\int_0^{\delta^*}\pare{\int_{V_{0,\delta}, \norm{v}_{L^2(\D)}\leq \epsilon} e^{B_\delta(v)+C_\eta\int_\D|v|^4}d\mu_{V_{0,\delta}}(v)}\delta^{-5}d\delta.
		\end{equation}
	\end{remark}
	
	\subsection{Higher order term}
	Following \ref{est: strengthened change of variable}, it is sufficient to estimate
	\begin{equation*}
		\int_{V_{0,\delta}, \norm{v}_{L^2(\D)}\leq \epsilon} e^{B_\delta(v)+C_\eta\int_\D|v|^4}d\mu_{V_{0,\delta}}(v).
	\end{equation*}
	Applying H\"{o}lder's inequality, we divide it into two part: quadratic term, $B_\delta$ part and higher order term, $\int_\D |v|^4$ part.
	\begin{equation*}
		\pare{\int_{V_{0,\delta}}e^{(1+\eta)B_\delta(v)}d\mu_{V_{0,\delta}}(v)}^{\frac{1}{1+\eta}} \pare{\int_{\norm{v}_{L^2(\D)}\leq \epsilon}e^{C_\eta\int_\D|v|^4}d\mu_{V_{0,\delta}}(v)}^{\frac{\eta}{1+\eta}}.
	\end{equation*}
	In this subsection, we deal with the higher order term part and leave the quadratic part to Section \ref{section: quadratic part}. 
	
	\begin{lemma}\label{lemma: integrability for higher order term}
		\begin{equation*}
			\int_{\norm{v}_{L^2(\D)}\leq \epsilon}e^{C_{\eta}\int_{\D}|v|^4}d\mu_{V_{0,\delta}}(v)  < \infty
		\end{equation*}
	\end{lemma}
	\begin{proof}
		Recall that
		\begin{equation*}
			V_{\theta,\delta}:=\brac{v\in L^2_{rad}(\D): \jbrac{v, -\Delta (e^{i\theta}\partial_{\delta}\bQ_{\delta})} = \jbrac{v, -\Delta (ie^{i\theta}\bQ_{\delta})} = 0}. 
		\end{equation*}
		Let $t_1, t_2$ be $\dot{H}^1$-normalized vectors of $\bQ_\delta, i\bQ_\delta$, respectively, that is
		\begin{equation*}
			t_1 = \frac{\partial_\delta \bQ_\delta}{\norm{\partial_\delta \bQ_\delta}_{\dot{H}^1}},\; t_2 = \frac{i\bQ_\delta}{\norm{i\bQ_\delta}_{\dot{H}^1}}.
		\end{equation*}
		Then the corresponding orthogonal projection $P_V$ on $V_{0,\delta}$ follows as
		\begin{equation*}
			P_{V_{\theta,\delta}} u := u - \sum_j \jbrac{u,(-\Delta)t_j}t_j.
		\end{equation*}
		Since $\mu_{V_{0,\delta}}$ is the free Gaussian measure on $V_{0,\delta}$, we have
		\begin{equation*}
			\int_{\norm{v}_2\leq \epsilon}e^{C_{\eta}\int_{\D}|v|^4}d\mu_{V_{0,\delta}}(v) = \int_{\norm{P_V u}_2\leq \epsilon}e^{C_\eta \int_\D |P_V u|^4}d\mu(u).
		\end{equation*}
		As in Proposition \ref{prop: outside ground state}, we further slice the domain:
		\begin{equation*}
			F_{2N}:=\brac{u\in L^2_{rad}(\D): \norm{P_{\geq M}P_V u}_4 > \lambda_M, M<2N; \norm{P_{\geq 2N}P_V u}_4 \leq \lambda_{2N}},
		\end{equation*}
		where $\lambda_N = (\log N)^{3/4}$. 
		Therefore,
		\begin{equation}\label{loc:expectation}
			\E_\mu \sqbrac{e^{C_\eta\int_\D |P_V u|^4}, F_{2N}, \norm{P_V u}_2\leq \epsilon}\leq e^{C_1\lambda_{2N}^4}\E_\mu\sqbrac{e^{C_2 \int_\D |P_{\leq N}P_V u|^4}, F_{2N}, \norm{P_V u}_2\leq \epsilon}.
		\end{equation}
		$C_1,C_2$ above depend only on $\eta$.
		Applying GNS inequality and conditioning on $\norm{P_V u}_2\leq \epsilon$,
		\begin{align*}
			\int |P_{\leq N}P_V u|^4 
			&\lesssim \norm{P_{\leq N}P_V u}_2^2\norm{P_{\leq N}P_V u}_{\dot{H}^1}^2 \leq \epsilon^2 \norm{P_{\leq N} u - \sum_j \jbrac{u,(-\Delta)t_j}P_{\leq N}t_j}_{\dot{H}^1}^2\\
			&\lesssim \epsilon^2 \pare{\norm{P_{\leq N}u}_{\dot{H}^1}^2+\sum_j \jbrac{u,(-\Delta)t_j}^2 \norm{P_{\leq N}t_j}_{\dot{H}^1}^2}\\
			&\leq \epsilon^2 \pare{\norm{P_{\leq N}u}_{\dot{H}^1}^2+\sum_j \jbrac{u,(-\Delta)t_j}^2}.
		\end{align*}
		(\ref{loc:expectation}) is then bounded by
		\begin{equation*}
			e^{C_1\lambda_{2N}^4}\E_\mu\sqbrac{e^{C_3\epsilon^2\norm{P_{\leq N} u}_{\dot{H}^1}^2}}^{1/4} \prod_j  \E_\mu\sqbrac{e^{C_4\epsilon^2 \jbrac{u,(-\Delta)t_j}^2}}^{1/4}\mu\pare{F_{2N}}^{1/4}
		\end{equation*}
		The first expectation has the same form as (\ref{eqn: exponential gaussian}), therefore equals to $\exp\pare{\frac{N}{2}\log\pare{\frac{1}{1-2C_3\epsilon^2}}}$.
		
		For the second expectation, note that $\jbrac{u,(-\Delta)t_j}$ is a mean-zero Gaussian. Using the series representation (\ref{def: series u precise}), its variance is bounded by $2\norm{t_j}_{\dot{H}^1}^2 = 2$. So, the expectation is bounded provided $\epsilon$ is small enough.
		
		Finally,
		\begin{align*}
			\mu(F_{2N}) 
			&\leq \mu\brac{\norm{P_{\geq N}P_V u}_4 \geq \lambda_N}\\
			&\leq \mu\brac{\norm{P_{\geq N}u}_4\geq \lambda_N/3} + \sum_j \mu\brac{\jbrac{u,(-\Delta)t_j}\norm{P_{\geq N}t_j}_4 \geq \lambda_N/3}.
		\end{align*}
		By (\ref{ineq:section2 F_N}) and $\lambda_N = (\log N)^{3/4}$, $\mu\brac{\norm{P_{\geq N}u}_4\geq \lambda_N/3} \leq \exp\pare{-c(\log N)N}$. Applying GNS inequality,
		\begin{equation*}
			\norm{P_{\geq N}t_j}_4 \lesssim \norm{P_{\geq N}t_j}_2^{1/2}\norm{P_{\geq N}t_j}_{\dot{H}^1}^{1/2}\leq \norm{P_{\geq N}t_j}_2^{1/2}.
		\end{equation*}
		Since $t_j$ are smooth functions, we have
		\begin{equation*}
			\norm{P_{\geq N}t_j}_2^2 = \sum_{n\geq N}\jbrac{t_j,e_n}^2 \leq\norm{(-\Delta)^k t_j}_2^2 \sum_{n\geq N}\frac{\norm{e_n}_2^2}{z_n^{4k}} \lesssim_k \frac{1}{N^{4k-1}}.
		\end{equation*}
		Thus,
		\begin{equation*}
			\mu\brac{\jbrac{u,(-\Delta)t_j}\norm{P_{\geq N}t_j}_4 \geq \lambda_N/3} \leq \mu\brac{\jbrac{u,(-\Delta)t_j} \gtrsim N} \leq \exp\pare{-cN^2}.
		\end{equation*}
		Now summing the factors above,
		\begin{equation*}
			\int_{\norm{v}_2\leq \epsilon}e^{C_{\eta}\int_{\D}|v|^4}d\mu_{V_{0,\delta}}(v) \lesssim \sum_N \exp\pare{C_1(\log 2N)^3 + \frac{N}{8}\log(\frac{1}{1-2C_3\epsilon^2}) -cN^2 } <\infty.
		\end{equation*}
	\end{proof}

	\subsection{Proof of Theorem \ref{thm:decomposition}}\label{section: normal bundle decomposition}
	Fix $\bar{\theta}, \bar{\delta}$, let $G$ be the diffeomorphism defined in the statement. The differentiability of $G$ is straightforward. To prove that $G$ is invertible, we use the following version of Inverse Function Theorem.
	\begin{theorem}\label{thm:inverse function theorem}
		Let $X, Y$ be Banach spaces, and $f: X \rightarrow Y$ be a $C^1$ maps. Suppose $df(x_0)$ is invertible and for some $\kappa<1$, $R>0$,
		\begin{align}\label{loc: inverse function}
			\norm{df(x_0)^{-1} \circ df(x)-Id_X}\leq \kappa, \quad x\in \bar{B}^X(x_0,R).
		\end{align}
		Then $f$ is invertible near $x_0$. Moreover, for $r:= \frac{1-\kappa}{\norm{df(x_0)^{-1}}}R$,  
		\begin{align*}
			\bar{B}^Y(f(x_0),r) \subset f(\bar{B}^X(x_0,R)).
		\end{align*}
	\end{theorem}
	\begin{proof}
		Assume $x_0=0$ and $f(0)=0$. Denote $A=df(0)$. For any $y\in \bar{B}^Y(0,r)$, define map $H_y : \bar{B}^X(0,R) \rightarrow X$ as
		\begin{equation*}
			H_y(x) :=  A^{-1}y-(A^{-1}f(x)-x).
		\end{equation*}
		By (\ref{loc: inverse function}), for $x_1,x_2 \in \bar{B}^X(0,R)$,
		\begin{align*}
			\norm{H(x_1)-H(x_2)} = \norm{\pare{A^{-1}f-Id}(x_1)-\pare{A^{-1}f-Id}(x_2)}\leq \kappa\norm{x_1-x_2}.
		\end{align*}
		Moreover,
		\begin{equation*}
			\norm{H_y(x)}_X \leq \norm{A^{-1}y}_X + \norm{A^{-1}f(x)-x}_X \leq \norm{A^{-1}}r+\kappa R \leq R.
		\end{equation*}
		
		So $H$ is a contraction map on $\bar{B}^X(0,R)$. Banach's Fixed-Point Theorem implies the existence a unique fixed point $x$ such that $H_y(x)=x$, that is $f(x)=y$.
	\end{proof}
	
	To apply the Inverse Function Theorem, we need operator bounds for
	\begin{equation*}
	    dG^{-1}(0,1,0)\circ dG(\theta, \delta, w) - Id.
	\end{equation*}
	In the rest of this subsection, we set, with out loss of generality, $\bar{\theta} = 0$.
	
	\begin{lemma}\label{lemma: inverse of dG}
		For $\bar{\delta} \gtrsim N^{-1}$, the derivative
		\begin{align*}
			dG(0, 1, 0): \R\times\R\times V_{0,\bar{\delta}}\cap E_N \longrightarrow L^2_{rad}(\D)\cap E_N
		\end{align*}
		is invertible and the inverse $dG^{-1}(0, 1, 0)$ is bounded uniformly in $\bar{\delta}$.
	\end{lemma}
	\begin{proof}
		Given $\mathbf{v} = (\alpha,\beta,v)\in \R \times \R \times V_{0,\bar{\delta}}\cap E_N$, denote $u:= dG(0, 1, 0)\mathbf{v}$. Direct computation gives,
		\begin{align}\label{loc,eqn: image of dG}
			u = \alpha P_{\leq N}i\bQ_{\bar{\delta}} + \beta\bar{\delta} P_{\leq  N}\partial_{\delta}\bQ_{\bar{\delta}} + v.
		\end{align}
		In coordinate $(P_{\leq N}i\bQ_{\bar{\delta}}, P_{\leq  N}\partial_{\delta}\bQ_{\bar{\delta}}, V_{0, \bar{\delta}}\cap E_N)$ of $L^2_{rad}(\D)\cap E_N$, the map $dG(0, 1, 0)$ has matrix representation
		\begin{equation*}
			\begin{pmatrix} 
				1 & 0 & 0\\
				0 & \bar{\delta} & 0\\
				0 & 0 & Id
			\end{pmatrix},
		\end{equation*}
		which is invertible.
		By orthogonality, (\ref{loc,eqn: image of dG}) yields
		\begin{align*}
			\alpha &= \frac{\jbrac{u,(-\Delta)i\bQ_{\bar{\delta}}}}{\jbrac{P_{\leq N}\bQ_{\bar{\delta}},(-\Delta)\bQ_{\bar{\delta}}}}\\
			\beta &=\bar{\delta}^{-1} \frac{\jbrac{u,(-\Delta)\partial_{\delta}\bQ_{\bar{\delta}}}}{\jbrac{P_{\leq N}\partial_{\delta}\bQ_{\bar{\delta}},(-\Delta)\partial_{\delta}\bQ_{\bar{\delta}}}}\\
			v &= u-\alpha P_{\leq N}i\bQ_{\bar{\delta}} -\beta\bar{\delta}P_{\leq N} \partial_{\delta}\bQ_{\bar{\delta}}.
		\end{align*}
		
		Next we will obtain lower bounds for the denominators. Since $\bQ_\delta$ differs $Q_\delta$ only by $O(e^{-c\delta^{-1}})$ (also true for its derivatives), it suffices to estimate the corresponding expression with direct truncation $Q_\delta$. Recall $e_n$ are $L^2$-normalized eigenfunctions. We have
		\begin{align*}
			\jbrac{Q_{\bar{\delta}},e_n} &= \int_\D \bar{\delta}^{-1}Q(\bar{\delta}^{-1}x)J_0(z_n x)\norm{J_0(z_n\cdot)}_{L^2(\D)}^{-1}dx \\
			&= \norm{J_0(z_n\cdot)}_{L^2(\D)}^{-1}\bar{\delta} \int_{\bar{\delta}^{-1}\D}Q(y)J_0(z_n\bar{\delta} y)dy\\
			&= \norm{J_0(z_n\cdot)}_{L^2(\D)}^{-1}\bar{\delta} \pare{\hat{Q}(z_n\bar{\delta})+O\pare{\exp(-c\bar{\delta}^{-1 })}}.
		\end{align*}
		Here, using that $J_0$ is bounded near $0$ and asymptotic approximation (\ref{est: Bessel asymtotic at infity}),
		\begin{align*}
			\hat{Q}(\xi) & := \int_{\R^2}Q(x)J_0(\xi x)dx \\
			&\lesssim \int_0^{\frac{1}{\xi}}Q(r)rdr +\int_{\frac{1}{\xi}}^{\infty}Q(r)\cos(\xi r-\frac{\pi}{4})(\xi r)^{-\frac{1}{2}}rdr \lesssim \jbrac{\xi}^{-2}.\\
			\norm{J_0(z_n\cdot)}_{L^2(\D)}^2 &= \frac{1}{z_n^2}\int_{z_n \D} J_0(y)^2 dy = \frac{1}{z_n^2} \int_0^{z_n} J_0(r)^2rdr \\
			&=\frac{1}{z_n^2}\pare{\int_0^1 O(1) dr + \int_1^{z_n} \cos(r-\frac{\pi}{4})^2 dr}\sim\frac{1}{z_n}.
		\end{align*}
		Combining with (\ref{est: roots of J_0})
		\begin{equation}\label{loc,size:lower bound 1}
			\begin{aligned}
				\jbrac{P_{\leq N}(-\bar{\delta})Q_{\bar{\delta}},Q_{\bar{\delta}}} &= \sum_{n\leq N} z_n^2 \jbrac{Q_{\bar{\delta}},e_n}^2 \sim \sum_{n\leq N} z_n^3\bar{\delta}^2 \pare{\hat{Q}(z_n\bar{\delta})^2 +O(\exp(-c\bar{\delta}^{-1}))}\\
				&\sim \bar{\delta}^{-2}\pi\sum_{n\leq N} (z_n\bar{\delta})^3 \hat{Q}(z_n\bar{\delta})^2 \cdot \bar{\delta}(z_n-z_{n-1})\\
				&\sim \bar{\delta}^{-2}\pi \int_0^{z_N \bar{\delta}}y^3\hat{Q}(y)^2 dy\\
				\pare{\text{for  }\bar{\delta} \geq \frac{1}{z_N} \sim \frac{1}{N}}\quad &\geq c_1\bar{\delta}^{-2}.
			\end{aligned}
		\end{equation}
		Similar computation yields, for $\bar{\delta} \gtrsim N^{-1}$,
		\begin{equation}\label{loc,size:lower bound 2}
			\jbrac{P_{\leq N}(-\Delta)\partial_{\delta}Q_{\bar{\delta}},\partial_{\delta}Q_{\bar{\delta}}}\gtrsim \bar{\delta}^{-4}
		\end{equation}
		
		Using (\ref{est: H^k norm of Q}), (\ref{loc,size:lower bound 1}), (\ref{loc,size:lower bound 2}) and Cauchy-Schwarz,\footnote{The implicit constant in the inequality depends on $\delta^*$}
		\begin{align*}
			\alpha &\lesssim \bar{\delta}^2 \norm{u}_{L^2(\D)}\norm{\bQ_{\bar{\delta}}}_{H^2(\D)} \lesssim \norm{u}_{L^2(\D)}.\\
			\beta &\lesssim \bar{\delta}^3 \norm{u}_{L^2(\D)}\norm{\partial_{\delta}\bQ_{\bar{\delta}}}_{H^2(\D)} \lesssim \norm{u}_{L^2(\D)}.\\
			\norm{v}_{L^2(\D)} &\lesssim \norm{u}_{L^2(\D)} + |\alpha| + |\beta| \lesssim  \norm{u}_{L^2(\D)}.
		\end{align*}
	\end{proof}
	
	\begin{lemma}\label{lemma: compare dG^{-1}dG with id}
		For $\bar{\delta} \gtrsim N^{-1}$, we have
		\begin{equation*}
			\norm{dG^{-1}(0,1,0)\circ dG(\theta, \delta, w) - Id}_{op} \lesssim |\theta|+|\delta-1|(\delta^{-1}+1)+\norm{w}_{L^2(\D)}. 
		\end{equation*}
	\end{lemma}
	\begin{proof}	
		Given $\mathbf{v} = (\alpha,\beta,v)\in \R \times \R \times V_{0,\bar{\delta}}\cap E_N$; denote the image by $\tilde{\mathbf{v}} = (\tilde{\alpha},\tilde{\beta}, \tilde{v})$ and the intermediate image by $u$. That is,
		\begin{align*}
			dG(\theta,\delta,w)\mathbf{v} = u = dG(0,1,0)\mathbf{\tilde{v}}.
		\end{align*}
		Direct computation yields,
		\begin{align*}
			dG(0,1,0)\mathbf{\tilde{v}} &= \tilde{\alpha} P_{\leq N}i\bQ_{\bar{\delta}} + \tilde{\beta}\bar{\delta} P_{\leq N}\partial_{\delta}\bQ_{\bar{\delta}} + \tilde{v},\\
			dG(\theta,\delta,w)\mathbf{v} &= \alpha P_{\leq N}ie^{i\theta}\bQ_{\bar{\delta}\delta}+ \beta\bar{\delta} P_{\leq N}(\partial_{\delta}\bQ)_{\bar{\delta}\delta} + P_{V_{0,\bar{\delta}\delta}\cap E_N}v +d\pare{P_{V_{0,\bar{\delta}\delta}\cap E_N}w}(\alpha,\beta).
		\end{align*}
		Here, $(\partial_{\delta}\bQ)_{\bar{\delta}\delta}$ means $\frac{\partial}{\partial_{s}}\bQ_s\mid_{s = \bar{\delta}\delta}$, and
		\begin{align*}
			d\pare{P_{V_{0,\bar{\delta}\delta}\cap E_N}w}(\alpha,\beta) 
			&= \pare{\alpha \partial_{\theta} + \beta \partial_{\delta}}\pare{-\frac{\jbrac{w,P_{\leq N}(-\Delta)ie^{i\theta}\bQ_{\bar{\delta}\delta}}}{\jbrac{P_{\leq N}(-\Delta)ie^{i\theta}\bQ_{\bar{\delta}\delta},ie^{i\theta}\bQ_{\bar{\delta}\delta}}}ie^{i\theta}\bQ_{\bar{\delta}\delta}}\\
			&+\pare{\alpha \partial_{\theta} + \beta \partial_{\delta}} \pare{-\frac{\jbrac{w,P_{\leq N}(-\Delta)\bar{\delta}e^{i\theta}(\partial_{\delta}\bQ)_{\bar{\delta}\delta}}}{\jbrac{P_{\leq N}(-\Delta)\bar{\delta}e^{i\theta}(\partial_{\delta}\bQ)_{\bar{\delta}\delta},\bar{\delta}e^{i\theta}(\partial_{\delta}\bQ)_{\bar{\delta}\delta}}}\bar{\delta}e^{i\theta}(\partial_{\delta}\bQ)_{\bar{\delta}\delta}}.
		\end{align*}
		Similar to the proof of Lemma \ref{lemma: inverse of dG}, applying (\ref{loc,size:lower bound 1}), (\ref{loc,size:lower bound 2}) and (\ref{est: H^k norm of Q}) we get\footnote{Heuristically, each $\Delta$ gives $\bar{\delta}^{-2}\delta^{-2}$, while each $\partial_{\delta}$ gives $\delta^{-1}$.}
		\begin{align}\label{loc, est: dP norm}
			\norm{d\pare{P_{V_{0,\bar{\delta}\delta}\cap E_N}w}(\alpha,\beta)}_{L^2(\D)} \lesssim \pare{|\alpha|+|\beta|}\norm{w}_{L^2(\D)}.
		\end{align}
		By orthogonality,
		\begin{align*}
			\tilde{\alpha}
			&= \frac{\jbrac{u,(-\Delta)i\bQ_{\bar{\delta}}}}{\jbrac{P_{\leq N}\bQ_{\bar{\delta}},(-\Delta)\bQ_{\bar{\delta}}}}\\
			&= \alpha\frac{\jbrac{P_{\leq N}ie^{i\theta}\bQ_{\bar{\delta}\delta},(-\Delta)i\bQ_{\bar{\delta}}}}{\jbrac{P_{\leq N}\bQ_{\bar{\delta}},(-\Delta)\bQ_{\bar{\delta}}}}
			+\beta\bar{\delta}\frac{\jbrac{P_{\leq N}e^{i\theta}(\partial_{\delta}\bQ)_{\bar{\delta}\delta},(-\Delta)i\bQ_{\bar{\delta}}}}{\jbrac{P_{\leq N}\bQ_{\bar{\delta}},(-\Delta)\bQ_{\bar{\delta}}}} \\ &\qquad+\frac{\jbrac{P_{V_{0,\bar{\delta}\delta}\cap E_N}v,(-\Delta)i\bQ_{\bar{\delta}}}}{\jbrac{P_{\leq N}\bQ_{\bar{\delta}},(-\Delta)\bQ_{\bar{\delta}}}} 
			+\frac{\jbrac{d\pare{P_{V_{0,\bar{\delta}\delta}\cap E_N}w}(\alpha,\beta),(-\Delta)i\bQ_{\bar{\delta}}}}{\jbrac{P_{\leq N}\bQ_{\bar{\delta}},(-\Delta)\bQ_{\bar{\delta}}}}.
		\end{align*}
		The first term is the dominant part. Applying Cauchy-Schwartz, (\ref{loc,size:lower bound 1}), (\ref{loc,size:lower bound 2}) and (\ref{est: H^k norm of Q}),
		\begin{align*}
			\frac{\jbrac{P_{\leq N}ie^{i\theta}\bQ_{\bar{\delta}\delta},(-\Delta)i\bQ_{\bar{\delta}}}}{\jbrac{P_{\leq N}\bQ_{\bar{\delta}},(-\Delta)\bQ_{\bar{\delta}}}} 
			&= 1 + \frac{\jbrac{P_{\leq N}\pare{ie^{i\theta}\bQ_{\bar{\delta}\delta}-i\bQ_{\bar{\delta}}},(-\Delta)i\bQ_{\bar{\delta}}}}{\jbrac{P_{\leq N}\bQ_{\bar{\delta}},(-\Delta)\bQ_{\bar{\delta}}}}\\
			&= 1+O\pare{\norm{e^{i\theta}\bQ_{\bar{\delta}\delta}-\bQ_{\bar{\delta}}}_{L^2(\D)}} + O(e^{-c\delta^{-1}}).
		\end{align*}
		Fundamental theorem of calculus gives
		\begin{align*}
			\norm{e^{i\theta}\bQ_{\bar{\delta}\delta}-\bQ_{\bar{\delta}}}_{L^2(\D)}
			&\leq |e^{i\theta}-1|\norm{\bQ_{\bar{\delta}\delta}}_{L^2(\D)} + \norm{\bQ_{\bar{\delta}\delta}-\bQ_{\bar{\delta}}}_{L^2(\D)}\\
			&\lesssim |\theta| + \int_{\bar{\delta}}^{\bar{\delta}\delta}\norm{(\partial_{\delta}\bQ)_s}_{L^2}ds\\
			&\lesssim |\theta| + |\ln \delta| \lesssim |\theta| + |\delta -1|.
		\end{align*}
		Then use orthogonality to estimate the 2nd and 3rd terms:
		\begin{align*}
			\frac{\jbrac{P_{\leq N}e^{i\theta}(\partial_{\delta}\bQ)_{\bar{\delta}\delta},(-\Delta)i\bQ_{\bar{\delta}}}}{\jbrac{P_{\leq N}\bQ_{\bar{\delta}},(-\Delta)\bQ_{\bar{\delta}}}} 
			&= \frac{\jbrac{P_{\leq N}\pare{e^{i\theta}(\partial_{\delta}\bQ)_{\bar{\delta}\delta}- \partial_{\delta}\bQ_{\bar{\delta}}},(-\Delta)i\bQ_{\bar{\delta}}}}{\jbrac{P_{\leq N}\bQ_{\bar{\delta}},(-\Delta)\bQ_{\bar{\delta}}}}\\
			&\lesssim \norm{e^{i\theta}(\partial_{\delta}\bQ)_{\bar{\delta}\delta}- \partial_{\delta}\bQ_{\bar{\delta}}}_{L^2(\D)}\\
			&\lesssim  |\theta| + \int_{\bar{\delta}}^{\bar{\delta}\delta}\norm{(\partial^2_{\delta}\bQ)_s}_{L^2}ds\\
			&\lesssim |\theta| + |\delta-1|\bar{\delta}^{-1}\delta^{-1},	
		\end{align*}
		and
		\begin{align*}
			\frac{\jbrac{P_{V_{0,\bar{\delta}\delta}\cap E_N}v,(-\Delta)i\bQ_{\bar{\delta}}}}{\jbrac{P_{\leq N}\bQ_{\bar{\delta}},(-\Delta)\bQ_{\bar{\delta}}}}
			&= \frac{\jbrac{P_{V_{0,\bar{\delta}\delta}\cap E_N}v,(-\Delta)\pare{i\bQ_{\bar{\delta}}-ie^{i\theta}\bQ_{\bar{\delta}\delta}}}}{\jbrac{P_{\leq N}\bQ_{\bar{\delta}},(-\Delta)\bQ_{\bar{\delta}}}}\\
			&\lesssim \norm{v}_{L^2}\cdot\pare{|\theta| + |\delta-1|}.
		\end{align*}
		For the last term, applying Cauchy-Schwartz,
		\begin{align*}
			\frac{\jbrac{d\pare{P_{V_{0,\bar{\delta}\delta}\cap E_N}w}(\alpha,\beta),(-\Delta)i\bQ_{\bar{\delta}}}}{\jbrac{P_{\leq N}\bQ_{\bar{\delta}},(-\Delta)\bQ_{\bar{\delta}}}} 
			&\lesssim \norm{d\pare{P_{V_{0,\bar{\delta}\delta}\cap E_N}w}(\alpha,\beta)}_{L^2(\D)}\\
			&\lesssim \pare{|\alpha|+|\beta|}\norm{w}_{L^2(\D)}.
		\end{align*}
		Thus
		\begin{align*}
			\tilde{\alpha} = \alpha + O\pare{\pare{|\theta| + |\delta-1||\delta^{-1}+1|}\cdot (\norm{v}_{L^2}+|\alpha|)+\pare{|\alpha|+|\beta|}\norm{w}_{L^2(\D)}}.
		\end{align*}
		Similarly,
		\begin{align*}
			\tilde{\beta} = \beta + O\pare{\pare{|\theta| + |\delta-1||\delta^{-1}+1|}\cdot (\norm{v}_{L^2}+|\beta|)+\pare{|\alpha|+|\beta|}\norm{w}_{L^2(\D)}}.
		\end{align*}
		Finally,
		\begin{align*}
			\tilde{v} - v 
			&= dG(\theta, \delta, w)\mathbf{v} -\tilde{\alpha} P_{\leq N}i\bQ_{\bar{\delta}} -\tilde{\beta}\bar{\delta} P_{\leq N}\partial_{\delta}\bQ_{\bar{\delta}} - v\\
			&= \alpha P_{\leq N}ie^{i\theta}\bQ_{\bar{\delta}\delta}+ \beta\bar{\delta} P_{\leq N}(\partial_{\delta}\bQ)_{\bar{\delta}\delta} +d\pare{P_{V_{0,\bar{\delta}\delta}\cap E_N}w}(\alpha,\beta)\\ &\quad+\pare{P_{V_{0,\bar{\delta}\delta}\cap E_N}v-v}-\tilde{\alpha} P_{\leq N}i\bQ_{\bar{\delta}} -\tilde{\beta}\bar{\delta} P_{\leq N}\partial_{\delta}\bQ_{\bar{\delta}}\\
			&= (\alpha - \tilde{\alpha})P_{\leq N}i\bQ_{\bar{\delta}} + \alpha P_{\leq N}(ie^{i\theta}\bQ_{\bar{\delta}\delta} -i\bQ_{\bar{\delta}}) +(\beta-\tilde{\beta})\bar{\delta}P_{\leq N}\partial_{\delta}\bQ_{\bar{\delta}}\\
			&\quad+\beta\bar{\delta}P_{\leq N}((\partial_{\delta}\bQ)_{\bar{\delta}\delta}-\partial_{\delta}\bQ_{\bar{\delta}})
			+d\pare{P_{V_{0,\bar{\delta}\delta}\cap E_N}w}(\alpha,\beta) +\pare{P_{V_{0,\bar{\delta}\delta}\cap E_N}v-v}.
		\end{align*}
		Using orthogonality, rewrite the last term as
		\begin{align*}
			P_{V_{0,\bar{\delta}\delta}\cap E_N}v-v 
			&= -\frac{\jbrac{v,P_{\leq N}(-\Delta)ie^{i\theta}\bQ_{\bar{\delta}\delta}}}{\jbrac{P_{\leq N}(-\Delta)ie^{i\theta}\bQ_{\bar{\delta}\delta},ie^{i\theta}\bQ_{\bar{\delta}\delta}}}ie^{i\theta}\bQ_{\bar{\delta}\delta}\\
			&\quad-\frac{\jbrac{v,P_{\leq N}(-\Delta)\bar{\delta}e^{i\theta}(\partial_{\delta}\bQ)_{\bar{\delta}\delta}}}{\jbrac{P_{\leq N}(-\Delta)\bar{\delta}e^{i\theta}(\partial_{\delta}\bQ)_{\bar{\delta}\delta},\bar{\delta}e^{i\theta}(\partial_{\delta}\bQ)_{\bar{\delta}\delta}}}\bar{\delta}e^{i\theta}(\partial_{\delta}\bQ)_{\bar{\delta}\delta}
			\\
			&= -\frac{\jbrac{v,P_{\leq N}(-\Delta)(ie^{i\theta}\bQ_{\bar{\delta}\delta}-i\bQ_{\bar{\delta}})}}{\jbrac{P_{\leq N}(-\Delta)ie^{i\theta}\bQ_{\bar{\delta}\delta},ie^{i\theta}\bQ_{\bar{\delta}\delta}}}ie^{i\theta}\bQ_{\bar{\delta}\delta}\\
			&\quad -\frac{\jbrac{v,P_{\leq N}(-\Delta)(e^{i\theta}(\partial_{\delta}\bQ)_{\bar{\delta}\delta}-\partial_{\delta}\bQ_{\bar{\delta}})}}{\jbrac{P_{\leq N}(-\Delta)e^{i\theta}(\partial_{\delta}\bQ)_{\bar{\delta}\delta},e^{i\theta}(\partial_{\delta}\bQ)_{\bar{\delta}\delta}}}e^{i\theta}(\partial_{\delta}\bQ)_{\bar{\delta}\delta}.
		\end{align*}
		Bound all these terms as before, using that $\bar{\delta}\ll 1$ and the estimates for $\tilde{\alpha}, \tilde{\beta}$,
		\begin{align*}
			\norm{\tilde{v} - v}_{L^2(\D)} &\lesssim |\alpha-\tilde{\alpha}| + |\alpha|\pare{|\theta|+|\delta-1|} + |\beta-\tilde{\beta}| + |\beta|\bar{\delta}\pare{|\theta|+|\delta-1|\bar{\delta}^{-1}\delta^{-1}}\\ 
			&\qquad+ \pare{|\alpha|+|\beta|}\norm{w}_{L^2} + \norm{v}_{L^2}\pare{|\theta|\bar{\delta}\delta+|\delta-1|}\\
			&\leq |\alpha-\tilde{\alpha}| + |\beta-\tilde{\beta}|  +  \pare{|\alpha|+|\beta|+\norm{v}_{L^2}}\pare{|\theta|+|\delta-1||\delta^{-1}+1|+\norm{w}_{L^2}}\\
			& \lesssim \pare{|\alpha|+|\beta|+\norm{v}_{L^2}}\pare{|\theta|+|\delta-1||\delta^{-1}+1|+\norm{w}_{L^2}}.
		\end{align*}
		To sum up,
		\begin{align*}
			\norm{dG^{-1}(0,1,0)\circ dG(\theta, \delta, 0)\mathbf{v} - \mathbf{v}}_{L^2(\D)}
			&\leq |\tilde{\alpha}-\alpha| + |\tilde{\beta}-\beta|+\norm{\tilde{v}-v}_{L^2}\\
			&\lesssim \norm{\mathbf{v}}\pare{|\theta|+|\delta-1|(\delta^{-1}+1)+\norm{w}_{L^2(\D)}}. 
		\end{align*}
	\end{proof}
	
	\begin{proof}[Proof of Theorem \ref{thm:decomposition}]
		Recall that we assume, without loss of generality, $\bar{\theta}=0$. We apply the Inverse Function Theorem \ref{thm:inverse function theorem} to $G$. By Lemma \ref{lemma: compare dG^{-1}dG with id},
		\begin{equation*}
			\kappa:=\norm{dG^{-1}(0,1,0)\circ dG(\theta, \delta, w) - Id}_{op} \leq C_0 \pare{ |\theta| + |\delta-1| +\norm{w}_{L^2}}.
		\end{equation*}	
		Define
		\begin{equation*}
			W = \brac{(\theta, \delta, w): |\theta| + |\delta-1| +\norm{w}_{L^2}<\min\brac{8\epsilon\norm{dG^{-1}(0,1,0)}, \frac{C_0^{-1}}{2}}}.
		\end{equation*}
		Then, for $(\theta,\delta,w)\in W$, $\kappa \leq \frac{1}{2}$. Since $\frac{1-\kappa}{\norm{dG^{-1}(0,1,0)}}\cdot 4\epsilon\norm{dG^{-1}(0,1,0)} \leq 2\epsilon$,
		\begin{equation*}
			\brac{u\in L^2_{rad}(\D)\cap E_N: \norm{u-P_{\leq N}\bQ_{\bar{\delta}}}_{L^2} < \epsilon} \subset G(W).
		\end{equation*}
	\end{proof}
	
	\begin{remark}\label{remark: L^2 norm of v}
		If we consider general $\bar{\theta}$, in the definition of $W$, $|\theta|$ has to be replaced by $|\theta-\bar{\theta}|$.
		Moreover, from the definition of $W$, for any decomposition $P_{\leq N}u = P_{\leq N}\pare{e^{i\theta}\bQ_{\delta}}+v$, we have $\norm{v}_{L^2}\leq  C\epsilon$. The constant $C$, by Lemma \ref{lemma: inverse of dG}, is independent of $\bar{\delta}$.
	\end{remark}
	
	\section{Quadratic part: a spectral analysis}\label{section: quadratic part}
	\subsection{Reduction to quadratic form}
	In this section, we estimate the quadratic part:
	\begin{equation*}
		\int_{V_{0,\delta}}e^{(1+\eta)B_\delta(v)}d\mu_{V_{0,\delta}}(v),
	\end{equation*}
	where
	\begin{equation*}
		B_\delta(v) = \delta^{-2}\jbrac{\bQ_{\delta},v} +\jbrac{\bQ_\delta^2, \frac{1}{2}v^2+(1+\eta)|v|^2}.
	\end{equation*}
	The strategy is to compare $B_\delta(v)$ with a simpler quadratic form. We illustrate our intuition on $H^1(\R^2)$. Using the notation in Remark \ref{remark: lagrange multiplier}, recall that $Q_\delta$ is a minimizer of $H_{\R^2}$ with constraint $M_{\R^2}(u) = \frac{1}{2}\norm{Q}_{L^2(\R^2)}^2$. A second derivative test yields
	\begin{lemma}\label{lemma: 2nd derivative test}
		For any $w \in H^1(\R^2)$ with $\jbrac{w, Q_\delta}=0$,
		\begin{equation}\label{ineq: precise 2nd derivative test}
			\frac{1}{2}\jbrac{-\Delta w,w} -\jbrac{Q_\delta^2, \frac{1}{2}w^2 + |w|^2} + \frac{\delta^{-2}}{2}\jbrac{w,w} \geq 0.
		\end{equation} 	
	\end{lemma}
	\begin{proof}
		For any $w$ with $\jbrac{w, Q_\delta}=0$, define a path $u(t) = \frac{\norm{Q}_2}{\norm{Q_\delta+tw}_2}(Q_\delta + tw)$. Then $u(t)$ lies in the constraint set $\{M_{\R^2}(u) = \frac{1}{2}\norm{Q}_2^2\}$, and $u(0) = Q_\delta$, $u_t(0) = w$. Thus, $H_{\R^2}(u(t))$ reaches its minimum at $t=0$. The second derivative test yields
		\begin{equation}\label{ineq: 2nd derivative test}
			0 \leq \frac{d^2}{dt^2}H_{\R^2}(u(t))\Big|_{t=0} = \jbrac{d^2H_{\R^2}(Q_\delta)w,w}+\jbrac{dH_{\R^2}(Q_\delta), u_{tt}(0)}.
		\end{equation}
		Since $M(u(t))$ is constant, we have
		\begin{equation*}
			0 =  \frac{d^2}{dt^2}M_{\R^2}(u(t))\Big|_{t=0} = \jbrac{d^2M_{\R^2}(Q_\delta)w,w}+\jbrac{dM_{\R^2}(Q_\delta), u_{tt}(0)}.
		\end{equation*}
		Recall that the Lagrange multiplier method gives $dH_{\R^2}(Q_\delta) - \lambda dM_{\R^2}(Q_\delta) = 0$. Thus,
		\begin{equation*}
			\jbrac{dH_{\R^2}(Q_\delta), u_{tt}(0)} = \lambda \jbrac{dM_{\R^2}(Q_\delta), u_{tt}(0)} = -\lambda\jbrac{d^2M_{\R^2}(Q_\delta)w,w}.
		\end{equation*}
		The second derivative test (\ref{ineq: 2nd derivative test}) now reads as
		\begin{equation*}
			d^2H_{\R^2}(Q_\delta) - \lambda d^2M_{\R^2}(Q_\delta) \geq 0.
		\end{equation*}
		We get the desired inequality by noting
		\begin{gather*}
			\jbrac{d^2H_{\R^2}(Q_\delta)w, w} = \jbrac{-\Delta w,w} -\jbrac{Q_\delta^2, w^2 + 2|w|^2},\\
			\jbrac{d^2M_{\R^2}(Q_\delta)w, w} = \jbrac{w,w},	
		\end{gather*}
		and $\lambda = -\delta^{-2}$ (ref. Remark \ref{remark: lagrange multiplier}).
	\end{proof}
	
	The first term $\frac{1}{2}\jbrac{-\Delta w,w} = \frac{1}{2}\int |\nabla w|^2$ is contained implicitly in the formal density of $\mu_{V_{0,\delta}}$, $e^{-\frac{1}{2}\int |\nabla w|^2}dw$ (despite the additional constraint $\jbrac{w, Q_\delta}=0$). The second derivative test suggests us to compare $B_\delta(v)$ with
	\begin{equation}\label{term: reference quadratic form}
		\jbrac{\bQ_\delta^2, \frac{1}{2}w^2 + |w|^2} - \frac{\delta^{-2}}{2}\jbrac{w,w}.
	\end{equation}
	
	There are two difficulties. First, the linear term $-\delta^{-2}\jbrac{\bQ_{\delta},v}$ cannot be bounded separately. Indeed, $\jbrac{\bQ_{\delta}, v}$ is a real valued Gaussian with mean 0 and variance $\jbrac{-\Delta^{-1}\bQ_{\delta},\bQ_{\delta}} \sim \delta^2$. Therefore
	\begin{equation*}
		\E_{\mu_{V_{0,\delta}}} \sqbrac{e^{-\delta^{-2}\jbrac{\bQ_{\delta}, v}}} \gtrsim e^{c\delta^{-2}},
	\end{equation*}
	which is too large. This issue will be addressed in Lemma \ref{lemma: reduction of linear term}. Second, the derivative test (\ref{ineq: precise 2nd derivative test}) holds with the additional constraint $\jbrac{w, Q_\delta}=0$, to overcome this, we need to single out the $Q_{\delta}$ direction. Define
	\begin{equation*}
		e:=\frac{P_{V_{0,\delta}}\pare{-\Delta^{-1}\bQ_{\delta}}}{\norm{P_{V_{0,\delta}}\pare{-\Delta^{-1}\bQ_{\delta}}}_{\dot{H}^1(\D)}} = \frac{-\Delta^{-1}\bQ_{\delta}}{\norm{-\Delta^{-1}\bQ_{\delta}}_{\dot{H}^1(\D)}}+O(e^{-c\delta^{-1}}).
	\end{equation*}
	Recall $\Delta^{-1}$ is the solution map of the Poisson equation $\Delta u = f$, $f\in L^2_{rad}(\D)$, with Dirichlet boundary condition. Accordingly, let $W_\delta$ be the subspace $\dot{H}^1$-orthogonal to $span\{e\}$. We have
	\begin{equation*}
		V_{0,\delta} = W_{\delta} \oplus_{\dot{H}^1} span\{e\}.
	\end{equation*} 
	Write $v = ge +w$ for some $g\in\R$, $w\in W_\delta$. We obtain
	\begin{align*}
		B_\delta(v)
		&= \delta^{-2}\jbrac{\bQ_\delta, ge} + \jbrac{\bQ_\delta^2, (\frac{3}{2}+\eta)(ge +\Re w)^2 + (\frac{1}{2}+\eta)(\Im w)^2}\\
		&\leq  (1+3\eta)\jbrac{\bQ_\delta^2, \frac{3}{2}(\Re w)^2 + \frac{1}{2}(\Im w)^2} +\delta^{-2}\jbrac{\bQ_\delta, ge} + C_\eta \jbrac{\bQ_\delta^2, (ge)^2}.
	\end{align*}
	Compare with the reference form (\ref{term: reference quadratic form}), we expect the following.
	
	\begin{lemma}\label{lemma: reduction of linear term} Given $\eta>0$, there exists $\epsilon^*=\epsilon^*(\eta)$ small enough, such that for all $\epsilon\leq \epsilon^*$,
		\begin{align*}
			\delta^{-2}\jbrac{\bQ_\delta, ge} + C_\eta \jbrac{\bQ_\delta^2, (ge)^2} \leq -(1-\eta)\frac{\delta^{-2}}{2}\jbrac{w,w} + O(e^{-c\delta^{-1}}).
		\end{align*}
	\end{lemma}
	\begin{proof}
		The proof is almost identical to Lemma 6.13 in \cite{oh2021optimal}, we rephrase it here to be self-contained.
		
		Since $\norm{\bQ_{\delta}}_{L^{\infty}}\sim \delta^{-1}$, it is equivalent to show
		\begin{equation}\label{eqn:section3 lemma9}
			\jbrac{\bQ_{\delta},ge} + C_{\eta}\norm{ge}_2^2 + \frac{1-\eta}{2}\norm{w}_2^2 \lesssim e^{-c\delta^{-1}}.
		\end{equation}
		By the normal bundle decomposition\footnote{In this proof, $\norm{\cdot}_2$ is short for $\norm{\cdot}_{L^2(\D)}$.},
		\begin{align*}
			\norm{Q}_{L^2(\R^2)}^2 
			&\geq \norm{\bQ_{\delta}+v}_2^2 =\norm{\bQ_{\delta}}_2^2 + 2\jbrac{\bQ_{\delta},v}+\norm{v}_2^2\\
			&=\norm{\bQ_{\delta}}_2^2 + 2\jbrac{\bQ_{\delta},ge}+\norm{v}_2^2\\
			&=\norm{\bQ_{\delta}}_2^2 + 2\jbrac{\bQ_{\delta},ge}+\norm{w}_2^2+2\jbrac{ge,w}+\norm{ge}_2^2.
		\end{align*}
		Since $\norm{Q}_{L^2(\R^2)}^2 - \norm{\bQ_{\delta}}_2^2 = O(e^{-c\delta^{-1}})$,
		\begin{equation}\label{eqn:section3 ge v}
			2\jbrac{\bQ_{\delta},ge}+\norm{v}_2^2 = O(e^{-c\delta^{-1}}).
		\end{equation}
		
		By the definition of $e$, we have
		\begin{equation}
			\norm{e}_2 \sim \norm{\bQ_{\delta}}_{H^{-1}} \sim \delta \sim \jbrac{\bQ_{\delta}, e}.
		\end{equation}
		
		\noindent\textbf{When $g\geq 0$}, by (\ref{eqn:section3 ge v}),
		\begin{gather*}
			\norm{ge}_2 \sim \jbrac{\bQ_{\delta},ge} = O(e^{-c\delta^{-1}}),\\
			\norm{w}_2 \leq \norm{v}_2 + \norm{ge}_2 \lesssim e^{-c\delta^{-1}},
		\end{gather*}
		which implies (\ref{eqn:section3 lemma9}).
		
		\noindent\textbf{When $g<0$}, according to Remark \ref{remark: L^2 norm of v}, $\norm{v}_2 \leq \epsilon$. (\ref{eqn:section3 ge v}) implies
		\begin{align*}
			\norm{ge}_2 \sim |\jbrac{\bQ_{\delta},ge}| = \frac{1}{2}\norm{v}_2^2 + O(e^{-c\delta^{-1}}) \leq \epsilon,
		\end{align*}
		and 
		\begin{equation*}
			\norm{v}_2^2 \leq 2|\jbrac{\bQ_\delta, ge}| + O(e^{-c\delta^{-1}}).
		\end{equation*}
		On the other hand,
		\begin{gather*}
			\norm{v}_2^2 = \norm{w}_2^2+2\jbrac{ge,w}+\norm{ge}_2^2
			\geq (1-\eta/2)\norm{w}_2^2 - K_\eta\norm{ge}_2^2.		
		\end{gather*}
		Thus,
		\begin{align*}
			2|\jbrac{\bQ_{\delta},ge}| \geq \norm{v}_2^2 -  O(e^{-c\delta^{-1}}) \geq (1-\eta/2)\norm{w}_2^2 - K_\eta\norm{ge}_2^2-  O(e^{-c\delta^{-1}}),
		\end{align*}
		that is
		\begin{align*}
			(1-\eta/2)\norm{w}_2^2 + 3C_\eta\norm{ge}_2^2 
			&\leq 2|\jbrac{\bQ_{\delta},ge}| + (K_\eta + 3C_\eta)\norm{ge}_2^2+ O(e^{-c\delta^{-1}})\\
			&\leq (2+O(\epsilon))|\jbrac{\bQ_{\delta},ge}| +O(e^{-c\delta^{-1}}).
		\end{align*}
		Choose $\epsilon = \epsilon(\eta)$ small enough, we can multiply the inequality by a suitable constant less than 1, such that
		\begin{align*}
			(1-\eta)\norm{w}_2^2 + 2C_\eta\norm{ge}_2^2 \leq 2|\jbrac{\bQ_{\delta},ge}| +O(e^{-c\delta^{-1}}).
		\end{align*}
		which gives (\ref{eqn:section3 lemma9}).
	\end{proof}
	
	\subsection{Spectral analysis}\label{section: spectrum analysis}
	So far, we have shown:
	\begin{equation*}
		B_\delta(v) \leq (1+3\eta)\jbrac{\bQ_\delta^2, \frac{3}{2}(\Re w)^2 + \frac{1}{2}(\Im w)^2}-(1-\eta)\frac{\delta^{-2}}{2}\jbrac{w,w} + O(e^{-c\delta^{-1}}).
	\end{equation*}
	
	In this section, $H^1_{rad,0}(\D)$ is a Hilbert space equipped with inner product $\jbrac{\cdot, \cdot}_{\dot{H}^1} = \jbrac{(-\Delta) \cdot, \cdot}$. Denote $\Re H^1_{rad,0}$ (resp. $\Im H^1_{rad,0}$) the subspace of real (resp. pure-imaginary) valued functions in $H^1_{rad,0}$. Accordingly, define the real and imaginary part of $W_\delta \cap H^1_{rad,0}(\D)$ as
	\begin{gather*}
		W_R:=W_R(\delta)=\brac{w \in \Re H^1_{rad,0}(\D): \jbrac{w, (-\Delta)\partial_\delta \bQ_\delta} = \jbrac{w, \bQ_\delta} = 0},\\
		W_I:=W_I(\delta)=\brac{w \in \Im H^1_{rad,0}(\D): \jbrac{w, (-\Delta)\bQ_\delta} = 0}.
	\end{gather*}
	Let $P_{W_R}$ (resp. $P_{W_I}$) be the $\dot{H}^1$ orthogonal projection on $W_R$ (resp. $W_I$). Then define operators on $H_{rad,0}^1(\D)$ as
	\begin{align*}
		A_1 &= P_{W_R}(-\Delta)^{-1}\pare{-(1+5\eta)\frac{3}{2}\bQ_{\delta}^2+\frac{\delta^{-2}}{2}}P_{W_R},\\
		A_2 &= P_{W_I}(-\Delta)^{-1}\pare{-(1+5\eta)\frac{1}{2}\bQ_{\delta}^2+\frac{\delta^{-2}}{2}}P_{W_I}.
	\end{align*}
	Clearly,
	\begin{equation*}
		B_\delta(v) \leq -(1-\eta)\jbrac{A_1\Re w,\Re w}_{\dot{H}^1}-(1-\eta)\jbrac{A_2\Im w, \Im w}_{\dot{H}^1} + O(e^{-c\delta^{-1}}).
	\end{equation*}
	At the end of this section, we will show
	\begin{proposition}\label{prop: integrability of quadratic part}
		For $\eta>0$ small enough,
		\begin{equation*}
			\int_{W_\delta}\exp\pare{-(1-\eta)\jbrac{A_1\Re w,\Re w}_{\dot{H}^1}-(1-\eta)\jbrac{A_2\Im w, \Im w}_{\dot{H}^1}} d\mu_{W_\delta}(w) \lesssim e^{-c\delta^{-1}}.
		\end{equation*}
		The implicit constant is independent of $\delta$.
	\end{proposition}
	
	By definition, $A_1, A_2$ are symmetric. Since $\norm{\bQ_\delta}_{\infty}\sim \delta^{-1}$, $\jbrac{A_iw,w}_{\dot{H}^1}\lesssim \delta^{-2}\norm{w}_{L^2(\D)}$. The Rellich–Kondrachov theorem then implies that $A_i$ are compact on $H^1_{rad,0}(\D)$. Since the range of $A_1$ (resp. $A_2$) is in $W_R$ (resp. $W_I$), it is also a compact operator on $W_R$ (resp. $W_I$). In particular, its spectrum consists of eigenvalues, with the only possible essential spectrum at $0$. Let $h_n\in W_R$ be a normalized eigenfunction of $A_1$ corresponding to eigenvalue $\lambda_n$. Then $\brac{h_n}$ forms an orthonormal basis of $\ker A_1$. By construction of free Gibbs measure,
	\begin{equation}\label{term: product form of the integral}
		\int e^{-(1-\eta)\jbrac{A_1 w,w}_{\dot{H}^1}}d\mu_{W_R}(w) = \prod_n \E\sqbrac{e^{-(1-\eta)\lambda_n g^2}} = \prod_n \pare{1+2(1-\eta)\lambda_n}^{-1/2},
	\end{equation}
	where $g$ is a real-valued standard Gaussian. In order to get a finite product, we first need to check $1+2(1-\eta)\lambda_n>0$.
	
	\begin{proposition}\label{prop: eigenvalue larger than -1/2}
		The smallest eigenvalue of $A_1$ (resp. $A_2$) is greater than $-\frac{1}{2}+\epsilon_0$, where $\epsilon_0 \rightarrow 0$ as $\delta \rightarrow 0$.
	\end{proposition}
	
	\begin{proof}
		Assume by contradiction that there exists a fixed $\epsilon_0>0$, and eigenfunctions $w_n\in W_{R}(\delta_n)$, such that $\delta_n\rightarrow 0$ but
		\begin{align*}
			\jbrac{A_1w_n,w_n}_{\dot{H}^1}\leq (-\frac{1}{2}+\epsilon_0)\jbrac{w_n,w_n}_{\dot{H}^1} .
		\end{align*}
		Then
		\begin{align*}
			\jbrac{\pare{-(1+5\eta)\frac{3}{2}\bQ_{\delta_n}^2+\frac{\delta_n^{-2}}{2}+\pare{\frac{1}{2}-\epsilon_0}(-\Delta)}w_n,w_n} \leq 0.
		\end{align*}
		Set $\tilde{w}_n = \delta_n w_n(\delta_n \cdot)$ and extend by 0 outside $\delta^{-1}\D$, then the inequality above implies
		\begin{align}\label{loc, ineq: Bn negative}
			\jbrac{\pare{-(1+5\eta)\frac{3}{2}\pare{Q-Q(\delta_n^{-1})}^2+\frac{1}{2}+(\frac{1}{2}-\epsilon_0)(-\Delta)}\tilde{w}_n,\tilde{w}_n}_{L^2(\R^2)}\leq 0.
		\end{align}
		Define
		\begin{equation*}
			W_{\R^2,R}(\delta): = \brac{u\in \Re L^2_{rad}\cap \dot{H}^1(\R^2): \jbrac{\mathbbm{1}_{\delta^{-1}\D}\cdot u, Q-Q(\delta^{-1})}_{L^2(\R^2)}= \jbrac{\mathbbm{1}_{\delta^{-1}\D}\cdot u,(-\Delta)\partial_{\delta}Q}_{L^2(\R^2)}=0},
		\end{equation*}
		and 
		\begin{equation*}
			B_n := P_{W_{\R^2,R}(\delta_n)}\pare{-(1+5\eta)\frac{3}{2}\pare{Q-Q(\delta_n^{-1})}^2\mathbbm{1}_{\delta^{-1}\D}+\frac{1}{2}+(\frac{1}{2}-\epsilon_0)(-\Delta)}P_{W_{\R^2,R}(\delta_n)},
		\end{equation*}
		as an operator on $L_{rad}^2(\R^2)$. Since $\pare{Q-Q(\delta_n^{-1})}^2\mathbbm{1}_{\delta^{-1}\D}$ is bounded, the Schr\"{o}dinger operator inside the parenthesis has essential spectrum on $[\frac{1}{2},\infty)$. Note that $B_n$ is a restriction of this Schr\"{o}dinger operator on a subspace, therefore the essential spectrum of $B_n$ is contained in $[\frac{1}{2},\infty)$.  Moreover,
		\begin{align*}
			\jbrac{B_n w,w}_{L^2(\R^2)}\geq -(1+5\eta)\frac{3}{2} \norm{Q-Q(\delta_n^{-1})}_{L^{\infty}}^2\jbrac{w,w}_{L^2(\R^2)}.
		\end{align*}
		$B_n$ is semi-bounded. Thus, we can let $\lambda_n$ be infimum of the spectrum. $\lambda_n$ is bounded from below. Using (\ref{loc, ineq: Bn negative}) and $\tilde{w}_n \in W_{\R^2,R}(\delta_n)$, we see that $\lambda_n\leq 0$, therefore it is an eigenvalue. Up to a subsequence, we may assume $\lambda_n\rightarrow \lambda \leq 0$. Define the limit operator
		\begin{equation*}
			B_R(\epsilon_0,\eta) := P_{W_{\R^2,R}}\pare{-(1+5\eta)\frac{3}{2}Q^2+\frac{1}{2}+(\frac{1}{2}-\epsilon_0)(-\Delta)}P_{W_{\R^2,R}},
		\end{equation*}
		with
		\begin{equation*}
			W_{\R^2,R}: = \brac{u\in \Re L^2_{rad}\cap \dot{H}^1(\R^2): \jbrac{u, Q}_{L^2(\R^2)}= \jbrac{u,(-\Delta)\partial_{\delta}Q}_{L^2(\R^2)}=0}.
		\end{equation*}
		We are going to derive a contradiction by showing that $B_R(\epsilon_0,\eta)\geq \theta>0$. By continuity, it is sufficient to prove this for $\epsilon_0=\eta=0$.
		
		The same argument for $A_2$ would require us to show
		\begin{equation*}
			B_I(0,0):=P_{W_{\R^2,I}}\pare{-\frac{1}{2}Q^2+\frac{1}{2}+\frac{1}{2}(-\Delta)}P_{W_{\R^2,I}}\geq \theta>0,
		\end{equation*}
		where
		\begin{equation*}
			W_{\R^2,I}:=\brac{u\in \Im L^2_{rad}(\R^2): \jbrac{u,(-\Delta) Q}_{L^2(\R^2)}=0}.
		\end{equation*}
		
		Using Lemma \ref{lemma: 2nd derivative test} with $\delta = 1$, for any real valued $w$, we have
		\begin{align*}
			\jbrac{\underbrace{\pare{\frac{1}{2}(1-\Delta)-\frac{3}{2}Q^2}}_{T_R}w,w}_{L^2(\R^2)} &\geq 0,\quad \text{for }\jbrac{w,Q}_{L^2(\R^2)}=0\\
			\jbrac{\underbrace{\pare{\frac{1}{2}(1-\Delta)-\frac{1}{2}Q^2}}_{T_I}w,w}_{L^2(\R^2)} &\geq 0.
		\end{align*}
		Then $B_R(0,0)$ (resp. $B_I(0,0)$) is the restriction of $T_R$ (resp. $T_I$) on $W_{\R^2,R}$ (resp. $W_{\R^2,I}$). We need to show $T_R$ (resp. $T_I$) $\geq \theta >0$ on $W_{\R^2,R}$ (resp. $W_{\R^2,I}$). By the min-max principle, the spectrum of $T_R$ (resp. $T_I$) is contained in $[0,\infty)$.
		
		Since $Q$ is a Schwartz function, $T_I$, as an operator on $L_{rad}^2(\R^2)$, has the same essential spectrum as $1-\Delta$, that is, $[1,\infty)$.
		
		The ground state equation $-\Delta Q-Q^3 + Q=0$ implies $T_IQ=0$. Using the radial variable $r$, the linear ODE 
		\[T_Iu = \frac{1}{2}(-\partial_{r}^2-\frac{1}{r}\partial_r+1-Q^2)u=0\]
		has Wronskian $W = \frac{1}{r}$. Thus, if there is any other linearly independent solution $u$, we must have, as $r \to 0$,
		\begin{equation*}
			u'(r) = \frac{1}{rQ(r)} + \frac{Q'(r)}{Q(r)}u(r) \sim \frac{1}{r} + o(1)u(r)
		\end{equation*}
		which implies
		\begin{equation*}
		    \norm{u}_{\dot{H}^1(\R^2)}^2 \gtrsim \int_0^1 \frac{1}{r} - o(1)\norm{u}_{L^2(\R)}^2=\infty.
		\end{equation*}
		So the eigenspace for the minimal eigenvalue $\lambda_1=0$ is $span\{Q\}$. We conclude that, on $W_{1,I}^0$, $T_I \geq \theta>0$. Indeed, the normal component of $Q$ to $(-\Delta) Q$, denoted as $Q_{\perp}$, has strictly smaller norm: $\norm{Q_{\perp}}_2<\norm{Q}_2$. Thus, for any $u\in W_{\R^2,I}$, the component on $Q$ direction has size
		\begin{align*}
			\frac{\jbrac{u,Q}_{L^2(\R^2)}}{\norm{Q}_2} = \frac{\jbrac{u,Q_{\perp}}_{L^2(\R^2)}}{\norm{Q}_2}\leq \norm{u}_2 \frac{\norm{Q_{\perp}}_2}{\norm{Q}_2} = c_0 \norm{u}_2, \qquad c_0\in (0,1).
		\end{align*}
		Hence, write $u = P_{Q}u + w$, where $\jbrac{w, Q}_{L^2(\R^2)}=0$, and let $\lambda_2$ be the second smallest point of the spectrum. We have
		\begin{align*}
			\jbrac{T_I u, u}_{L^2(\R^2)} = \jbrac{T_I w, w}_{L^2(\R^2)} \geq \lambda_2 \norm{w}_2^2 \geq\lambda_2 (1-c_0^2)\norm{u}_2^2.
		\end{align*}
		
		By the similar argument, it suffices to show that $P_{Q^{\perp}}T_RP_{Q^{\perp}} u =0$ for $u \perp_{L^2(\R^2)} W_{\R^2,R}$.
		\begin{lemma}\label{lemma: 0 kernel}
			$T_R$, as an operator on $L^2_{rad}(\R^2)$ has trivial kernel: $\ker T_R = \{0\}$.
		\end{lemma}
		\begin{proof}
			It is known (see for example  \cite{Frank2014GROUNDSO}) that the corresponding $L^2(\R^2)$ operator $L^+:=-\Delta+1-3Q^2$ has $\ker L^+ = \text{span}\{\partial_{x_1}Q,\partial_{x_2}Q\}$. Clearly $\ker T_R \subset \ker L^+$. If there is some $u\in \ker T_R\setminus \brac{0}$, then by normalizing, we may assume $u = \partial_v Q$, for some unit vector $v\in \R^2$. Moreover, since $Q$ is radial, we may further rotate $u$ to get $u=\partial_{x_1}Q$. Since $u$ is radial, $u(x,\cdot)$ would be even. But $Q(x,\cdot)$ is also even and therefore its derivative, $u(x,\cdot) = \partial_{x_1}Q(x,\cdot)$, has to be odd. This forces $u \equiv 0$, a contradiction.
		\end{proof}
		
		Thus $T_{R} u =0$, if and only if $u=0$. By differentiating (\ref{eqn: scaling ground state}) with respect to $\delta$, we obtain $T_R(\partial_{\delta}Q) = Q$. Combining with Lemma \ref{lemma: 0 kernel} yields that $P_{Q^{\perp}}T_RP_{Q^{\perp}} u =0$ if and only if $u \in span\{Q, \partial_{\delta}Q\}$, which is orthogonal to $W_{\R^2,R}$.
	\end{proof}
	
	Next, we will show the eigenvalues in (\ref{term: product form of the integral}) is summable.
	\begin{proposition}\label{prop: asymptotic lower bound for eigenvalues}
		List the non-zero eigenvalues of $A_1$(or $A_2$) as $\lambda^-_1\leq \lambda^-_2 \leq \cdots < 0 < \cdots \leq \lambda^+_2 \leq \lambda^+_1$. Then
		\begin{equation}\label{est: asymptotic estimate for eigenvalues}
			\lambda^+_n \gtrsim \frac{\delta^{-2}}{n^2}, \; \lambda^-_n \gtrsim -\frac{1}{n^2}.
		\end{equation}
	\end{proposition}
	
	\begin{proof}
		We only prove this result for $A_1$, same argument applies for $A_2$ case.
		
		Define $\tilde{T}_R := (-\Delta)^{-1}\pare{-(1+5\eta)\frac{3}{2}\bQ_{\delta}^2+\frac{\delta^{-2}}{2}}$ as an operator on $H^1_{rad,0}(\D)$. Then $A_1 =P_{W_R}\tilde{T}_RP_{W_R}$ is a finite rank perturbation of $\tilde{T}_R$, thus the eigenvalues of $A_1$ are interlaced with eigenvalues of $\tilde{T}_R$. It is sufficient to get the same estimates (\ref{est: asymptotic estimate for eigenvalues}) for $B_R$, whose non-zero eigenvalues we list as $\bar{\lambda}^-_1 \leq \cdots < 0 < \cdots \leq \bar{\lambda}^+_1$.
		
		The strategy of the proof is to compare $B_R$ simpler Schr\"{o}dinger operators. We investigate the positive and negative eigenvalues separately.
		
		\paragraph{Positive eigenvalues:} By Lemma \ref{lemma: ground state}, we can choose some constant $a>0$ large enough such that $Q(r)\leq 1/4$ for $r\geq a$. Let
		\begin{align*}
			S_+(r) =
			\begin{bracket}
				&``-\infty", & r\leq a\delta\\
				&\frac{1}{4}\delta^{-2}, & r>a\delta
			\end{bracket},
		\end{align*}
		Then $S_+ \leq -(1+5\eta)\frac{3}{2}\bQ_{\delta}^2+\frac{\delta^{-2}}{2}$.
		Define $(-\Delta)^{-1}S_+$ as an operator on the subspace
		\begin{equation*}
			X_+:=\brac{u\in H^1_{rad,0}(\D): u(r)=0 \text{ for } r\in [0,a\delta]}.
		\end{equation*}
		Clearly, $(-\Delta)^{-1}S_+$ is a compact operator, and we list its positive eigenvalues as $0< \cdots \leq \mu_2 \leq \mu_1$. By the min-max principle,
		\begin{align*}
			-\bar{\lambda}^+_n &= \min_{\dim L = n}\max_{w\in L}\brac{\frac{\jbrac{-T_Rw,w}_{\dot{H}^1}}{\norm{w}_{\dot{H}^1}^2}}\\
			& \leq \min_{\dim L = n}\max_{w\in L}\brac{\frac{\jbrac{-S_+w,w}}{\norm{w}_{\dot{H}^1}^2}: w\in X_+}\cup\brac{\infty: w\notin X_+}\\
			& = \min_{\dim L = n, L\subset X_+}\max_{w\in L}\brac{\frac{\jbrac{-S_+w,w}}{\norm{w}_{\dot{H}^1}^2}: w\in X_+} = -\mu_n.
		\end{align*}
		
		The eigenvalue equation $(-\Delta)^{-1}S_+ f  = \mu f$ can be written as
		\begin{equation*}
			\pare{\partial_r^2+\frac{1}{r}\partial_r}f= -\frac{1}{4}\frac{\delta^{-2}}{\mu}f=, \; r\in [a\delta, 1],
		\end{equation*}
		with boundary value condition $f(a\delta) = f(1) =0$.
		
		As discussed in Section \ref{section: Bessel functions}, equation $(\partial_r^2 + \frac{1}{r}\partial_r)u = -w^2u$ has a general solution of the form $c_1J_0(wr) + c_2Y_0(wr)$. Denote $B=\sqrt{|\frac{1}{4}\frac{\delta^{-2}}{\mu}|}$. Since $\mu>0$, $B\in (0,\infty)$. The boundary value conditions yield
		\begin{align*}
			0 &= f(1) = c_1J_0(B) + c_2Y_0(B),\\
			0 &= f(a\delta) = c_1J_0(Ba\delta) + c_2Y_0(Ba\delta).
		\end{align*}
		The equations above yields the cross-product relation
		\begin{equation}\label{eqn: cross product, positve}
			J_0(B)Y_0(Ba\delta)-Y_0(B)J_0(Ba\delta) = 0.
		\end{equation}
		Let $B_n$ denotes the $n$th root of (\ref{eqn: cross product, positve}). It is known (see \cite{cochran1964remarks}) that\footnote{For more results on zeros of cross-product of Bessel function, see Section 10.21 in \cite{nist} and references wherein.} $B_n\sim \frac{n\pi}{1-a\delta}$.
		
		Since $B\sim \frac{\delta^{-1}}{\sqrt{\mu}}$, we have $\mu_n \sim \frac{\delta^{-2}(1-a\delta)^2}{n^2}$. Hence
		\begin{equation*}
			\bar{\lambda}_n^+\gtrsim \frac{\delta^{-2}}{n^2}.
		\end{equation*}
		
		\paragraph{Negative eigenvalues:} Using the same constant $a$ as in positive eigenvalue part, let
		\begin{equation*}
			S_-(r) =
			\begin{bracket}
				&-3\delta^{-2}, & r\leq 5\delta\\
				&``-\infty", & r> 5\delta
			\end{bracket}.
		\end{equation*}
		By Lemma \ref{lemma: ground state}, $\norm{Q}_{\infty} \leq \sqrt{2}$. Thu,  $S_- \leq -(1+5\eta)\frac{3}{2}\bQ_{\delta}^2+\frac{\delta^{-2}}{2}$. Then the operator $(-\Delta)^{-1}S_-$ is defined on subspace
		\begin{equation*}
			X_-:=\brac{u\in H^1_{rad}(B): u(r)=0 \text{ for } r\in [a\delta, 1]}.
		\end{equation*}
		List the negative eigenvalues of $(-\Delta)^{-1}S_-$ as $\nu_1 \leq \nu_2 \leq \cdots <0$. The min-max principle yields $\bar{\lambda}^-_n \geq \nu_n$.
		
		The eigenvalue equation $(-\Delta)^{-1}S_- f = \nu f$ can be written as
		\begin{equation*}
			\pare{\partial_r^2+\frac{1}{r}\partial_r}f= \frac{3\delta^{-2}}{\nu}f,  r\leq a\delta,
		\end{equation*}
		with boundary condition $f(a\delta) = 0$. Since $f\in H^1_{rad,0}(\D)$, in radial variable $r$, we impose initial conditions $f'(0) = 0$ and normalized with $f(0) = 1$.
		
		Denote $A = \sqrt{|\frac{3\delta^{-2}}{\nu}|}$. Since $\nu<0$, $A\in (0,\infty)$, the initial condition $f'(0) = 0, f(0) = 1$ implies
		\begin{equation*}
			f(r) = J_0(Ar), \; r\leq a\delta.
		\end{equation*}
		Using $f(a\delta) = 0$, we obtain
		\begin{equation}\label{eqn: cross product, negative}
			f(a\delta) = J_0(Aa\delta) = 0.
		\end{equation}
		Let $A_n$ be the $n$th root of (\ref{eqn: cross product, negative}). By Theorem 7.2.1 in \cite{beals2010special}, $A_n\sim \frac{n\pi}{a\delta}$. Since $A\sim \frac{\delta^{-1}}{\sqrt{|\nu|}}$, we have $\nu \sim -\frac{1}{n^2}$. Hence:
		\begin{equation*}
			\bar{\lambda}_n^-\lesssim -\frac{1}{n^2}.
		\end{equation*}
	\end{proof}
	
	\begin{proof}[Proof of Proposition \ref{prop: integrability of quadratic part}]
		For $j=1,2$, list the eigenvalues of $A_j$ as $\lambda_{j,1}^- \leq \lambda_{j,2}^- \leq \cdots < 0 \leq \cdots \lambda_{j,2}^+ \leq \lambda_{j,1}^+$, and denote $h_{j,n}^{\pm}$ as the $\dot{H}$-normalized eigenfunction corresponding to eigenvalue $\lambda_{j,n}^{\pm}$. Since $W_\delta\cap H_{rad,0}^1(\D) = W_R \oplus W_I$ is a dense subspace, $\brac{h_{1,n}^{\pm}, ih_{2,n}^{\pm}}$ forms an orthogonal basis of $W_\delta$. The free Gaussian measure $\mu_{W_\delta}$ is then the law of random function
		\begin{equation*}
			w = \sum_{\nu\in\brac{+,-}}\sum_{j\in\brac{1,2}}\sum_n g_{j,n}^{\nu}h_{j,n}^{\nu},
		\end{equation*}
		where $g_{j,n}^{\nu}$ are independent real-valued standard Gaussian. Therefore,
		\begin{equation}\label{loc, term: quadratic integral}
			\begin{aligned}
				\int_{W_\delta}&\exp\pare{-(1-\eta)\jbrac{A_1\Re w,\Re w}_{\dot{H}^1}-(1-\eta)\jbrac{A_2\Im w, \Im w}_{\dot{H}^1}} d\mu_{W_\delta}(w)\\
				& =\prod_{\nu\in\brac{+,-}}\prod_{j\in\brac{1,2}}\prod_n \E\brac{e^{-(1-\eta)\lambda_{j,n}^{\nu}(g_{j,n}^{\nu})^2}}\\
				& = \prod_{\nu\in\brac{+,-}}\prod_{j\in\brac{1,2}}\prod_n \pare{1+2(1-\eta)\lambda_{j,n}^{\nu}}^{-\frac{1}{2}}.
			\end{aligned}
		\end{equation}
		Using Proposition \ref{prop: eigenvalue larger than -1/2} and Proposition \ref{prop: asymptotic lower bound for eigenvalues}, we get
		\begin{align*}
			\ln\pare{\prod_n \pare{1+2(1-\eta)\lambda_{j,n}^{+}}^{-\frac{1}{2}}}
			&= \sum_n -\frac{1}{2}\ln\pare{1+2(1-\eta)\lambda_{j,n}^{+}} \\
			&\lesssim -(1-\eta) \sum_n \lambda_{j,n}^{+}\\
			&\lesssim -(1-\eta) \sum_{n\geq \delta^{-1}}\frac{\delta^{-2}}{n^2} \leq -c_\eta \delta^{-1},
		\end{align*}
		whilst
		\begin{align*}
			\ln\pare{\prod_n \pare{1+2(1-\eta)\lambda_{j,n}^{-}}^{-\frac{1}{2}}}
			\lesssim (1-\eta) \sum_n\frac{1}{n^2} \leq C_\eta.
		\end{align*}
		Hence (\ref{loc, term: quadratic integral}) is $O_\eta(e^{-c\delta^{-1}})$.
	\end{proof}
	
	\section{Proof of the main result}
	We now summarize what has been proved to conclude Theorem \ref{thm: main result}. Corollary \ref{cor: integrability near soliton manifold} shows
	\begin{equation*}
		\E_\mu\sqbrac{e^{\frac{1}{4}\int_{\D} |u|^4}, U_{\epsilon}(0,\delta^*) ^C}<\infty.
	\end{equation*}
	Applying Lemma \ref{lemma: integrability for higher order term} and Proposition \ref{prop: integrability of quadratic part}, Proposition \ref{prop: change of variable} yields
	\begin{align*}
		\E_\mu\sqbrac{e^{\frac{1}{4}\int_{\D} |u|^4}, U_{\epsilon}(0,\delta^*)} \lesssim \int_0^{\delta^*}e^{-c \delta^{-1}}\delta^{-5}d\delta <\infty.
	\end{align*}
	In summary, we have proved normalizability of Gibbs measure $\rho$.
	
	\printbibliography
\end{document}